\documentclass[12pt,a4paper,reqno]{amsart}
\usepackage{amsmath}
\usepackage{geometry}
\usepackage{amssymb,latexsym}
\usepackage{verbatim}
\usepackage{extarrows}
\usepackage{enumerate}
\usepackage{txfonts}
\usepackage{mathtools}

\usepackage{mathtools}
\usepackage[tableposition=top]{caption}
\usepackage{booktabs,dcolumn}

\theoremstyle{plain}

\newtheorem{theorem}{Theorem}[section]
\newtheorem{proposition}[theorem]{Proposition}

\newtheorem{lemma}[theorem]{Lemma}
\newtheorem{corollary}[theorem]{Corollary}
\newtheorem{conjecture}[theorem]{Conjecture}

\theoremstyle{definition}

\newtheorem{definition}[theorem]{Definition}
\newtheorem{remark}[theorem]{Remark}

\newtheorem{example}[theorem]{Example}

\newcommand\Z{\mathbb{Z}}
\newcommand\Q{\mathbb{Q}}
\newcommand\R{\mathbb{R}}
\newcommand\T{\mathbb{T}}

\newcommand\N{\mathbb{N}}

\newcommand\B{\mathcal{B}}
\newcommand\X{\mathcal{X}}
\newcommand\Y{\mathcal{Y}}
\newcommand\Scal{\mathcal{S}}

\newcommand\op{{\operatorname{op}}}
\newcommand\Hom{\operatorname{Hom}}
\newcommand\Cond{\operatorname{Cond}}
\newcommand\Aut{\operatorname{Aut}}
\newcommand\Baire{{\mathcal{B}a}}
\newcommand\AbsMes{\mathbf{AbsMbl}}
\newcommand\SigmaAlg{{\mathbf{Bool}_\sigma}}
\newcommand\SigmaAlgop{{\mathbf{Bool}^\op_\sigma}}
\newcommand\ConcMes{\mathbf{CncMbl}}

\newcommand\eps{\varepsilon}

\begin{document}

\title[Uncountable Moore--Schmidt]{An uncountable Moore--Schmidt theorem}

 \author{Asgar Jamneshan}
 \address{Department of Mathematics, Koc University}
 \email{ajamneshan@ku.edu.tr}
 \author{Terence Tao}
 \address{Department of Mathematics, UCLA}
 \email{tao@math.ucla.edu}

\begin{abstract} 
We prove an extension of the Moore--Schmidt theorem on the triviality of the first cohomology class of cocycles for the action of an arbitrary discrete group on an arbitrary measure space and for cocycles with values in an arbitrary  compact Hausdorff abelian group.  
The proof relies on a ``conditional'' Pontryagin duality for spaces of abstract measurable maps.   
\end{abstract}

\maketitle

\setcounter{tocdepth}{1}



\section{Introduction} \label{intro}

\subsection{The countable Moore--Schmidt theorem}

Suppose that $X = (X,\Sigma_X,\mu)$ is a probability space, thus $\Sigma_X$ is a $\sigma$-algebra on $X$ and $\mu \colon \Sigma_X \to [0,1]$ is countably additive with $\mu(X)=1$.  If $Y = (Y,\Sigma_Y)$ is a measurable space and $f \colon X \to Y$ is a measurable map, we define the \emph{pullback map} $f^*: \Sigma_Y \to \Sigma_X$ by
$$ f^* E \coloneqq f^{-1}(E)$$
for $E \in \Sigma_Y$, and then define the \emph{pushforward measure} $f_* \mu$ on $Y$ by the usual formula
$$ f_*\mu(E) \coloneqq \mu(f^* E).$$
For reasons that will become clearer later, we will refer to measurable spaces and measurable maps as \emph{concrete measurable spaces} and \emph{concrete measurable maps} respectively; this creates a category $\ConcMes$.
We define $\Aut(X,\X,\mu)$ to be the space of all concrete invertible bimeasurable maps $T\colon X \to X$ such that $T_* \mu = \mu$; this is a group.
If $\Gamma = (\Gamma,\cdot)$ is a discrete group, we define a \emph{(concrete) measure-preserving action} of $\Gamma$ on $X$ to be a group homomorphism $\gamma \mapsto T^\gamma$ from $\Gamma$ to $\Aut(X,\X,\mu)$.  If $K = (K,+)$ is a compact Hausdorff\footnote{It is likely that the arguments here extend to non-Hausdorff compact groups by quotienting out the closure of the identity element, but the Hausdorff case already captures all of our intended applications and so we make this hypothesis to avoid some minor technical issues.} abelian group, which we endow with the Borel $\sigma$-algebra $\Sigma_K = \B(K)$, we define a \emph{$K$-valued (concrete measurable) cocycle} for this action to be a family $\rho = (\rho_\gamma)_{\gamma \in \Gamma}$ of concrete measurable maps $\rho_\gamma \colon X \to K$ such that for any $\gamma_1, \gamma_2 \in \Gamma$, the cocycle equation
    \begin{equation}\label{cocy}
     \rho_{\gamma_1\gamma_2} = \rho_{\gamma_1} \circ T^{\gamma_2} + \rho_{\gamma_2}
    \end{equation}
holds $\mu$-almost everywhere.  A cocycle $\rho$ is said to be a \emph{(concrete measurable) coboundary} if there exists a concrete measurable map $F: X \to K$ such that for each $\gamma \in \Gamma$, one has
\begin{equation}\label{cobound}
\rho_\gamma = F \circ T^\gamma - F
\end{equation}
 $\mu$-almost everywhere.  Note that \eqref{cobound} (for all $\gamma$) automatically implies \eqref{cocy} (for all $\gamma_1,\gamma_2$), although the converse does not hold in general.

It is of interest to determine the space of all $K$-valued concrete measurable coboundaries.  The following remarkable result of Moore and Schmidt \cite[Theorem 4.3]{moore1980coboundaries} reduces this problem to the case of coboundaries taking values in the unit circle $\T=\R/\Z$, at least under certain regularity hypotheses on the data  $\Gamma,X,K$.  More precisely, let $\hat K$ denote the Pontryagin dual of the compact Hausdorff abelian group $K$, that is to say the space of all continuous homomorphisms $\hat k: k \mapsto \langle \hat k, k \rangle$ from $K$ to  $\T$.

\begin{theorem}[(Countable) Moore--Schmidt theorem]\label{mst}
Let $\Gamma$ be a discrete group acting (concretely) on a probability space $X = (X,\Sigma_X,\mu)$, and let $K$ be a compact Hausdorff abelian group.  Assume furthermore:
\begin{itemize}
    \item[(a)] $\Gamma$ is at most countable.
    \item[(b)] $X = (X,\Sigma_X,\mu)$ is a standard Lebesgue space (thus $X$ is a Polish space, $\Sigma_X$ is the Borel $\sigma$-algebra, and $\mu$ is a probability measure on $\Sigma_X$).
    \item[(c)] $K$ is metrisable.
\end{itemize}
Then a $K$-valued concrete measurable cocycle $\rho = (\rho_\gamma)_{\gamma \in \Gamma}$ on $X$ is a coboundary if and only if the $\T$-valued cocycles $\langle \hat k, \rho\rangle \coloneqq (\langle \hat{k}, \rho_\gamma\rangle)_{\gamma \in \Gamma}$ are coboundaries for all $\hat{k} \in \hat K$. 
\end{theorem}

In fact, the results in \cite{moore1980coboundaries} extend to the case when $\Gamma$ and $K$ are locally compact groups (which are now assumed to be second countable instead of countable), and $(\langle \hat{k}, \rho_\gamma\rangle)_{\gamma \in \Gamma}$ is only assumed to be a coboundary for almost all $\hat{k} \in K$ with respect to some ``full'' measure.  We will not discuss such extensions of this theorem here, but mention that the original proof by Moore and Schmidt at this level of generality crucially relies on measurable selection theorems.     

The Moore--Schmidt theorem is a beautiful classification result which serves as a relevant technical tool in ergodic theory and probability.  
It formulates a condition for the triviality of the first cohomology class of cocycles - an important invariant of measure-theoretic actions of groups - by describing the size of the set of characters necessary and sufficient to test triviality.  
It is particularly helpful for understanding  the structure of cocycles. See e.g., \cite{host2005nonconventional,bergelson2010inverse,austin2015pleasant1} for applications in the structure theory of nonconventional ergodic averages of multiple recurrence type, \cite{aaronson2001local,gouezel2005berry} for applications to limit theorems in probability, and \cite{schmidt1980asymptotically,moore1979groups,bergelson2014rigidity,helson1986cocycles} for some applications in other classification and asymptotic results in ergodic theory. 

We briefly sketch here a proof of Theorem \ref{mst}.  Using the ergodic decomposition \cite{furstenberg2014recurrence} (which takes advantage of the hypotheses (a), (b)) we may assume without loss of generality that the action is ergodic.  By definition, for each $\hat k \in \hat K$ there exists a realization $\alpha_{\hat k}$ of an element of the group $L^0(X; \T)$ of concrete measurable functions from $X$ to $\T$, modulo $\mu$-almost everywhere equivalence, such that
\begin{equation}\label{hatk}
 \langle \hat{k}, \rho_\gamma \rangle = \alpha_{\hat k} \circ T^{\gamma} - \alpha_{\hat k}
 \end{equation}
$\mu$-almost everywhere.  For any $\hat k_1, \hat k_2 \in \hat K$, one sees from comparing \eqref{hatk} for $\hat k_1, \hat k_2, \hat k_1 + \hat k_2$ that the function $\alpha_{\hat k_1 + \hat k_2} - \alpha_{\hat k_1} - \alpha_{\hat k_2}$ is $\Gamma$-invariant up to $\mu$-almost sure equivalence, and hence equal in $L^0(X; \T)$ to a constant $c(\hat k_1, \hat k_2) \in \T$, by the ergodicity hypothesis.  Viewing $\T$ as a divisible\footnote{That is, for any $x \in \T$ and $n \in \N$, there exists $y \in \T$ such that $ny=x$.}  subgroup of the abelian group $L^0(X; \T)$, a routine application of Zorn's lemma\footnote{We freely assume the axiom of choice in this paper.} (see e.g.,  \cite[p.~46--47]{halmos2013lectures}) then lets us obtain a retract homomorphism $w: L^0(X; \T) \to \T$.   If we define the modified function $\tilde \alpha_{\hat k} \coloneqq \alpha_k - w(\alpha_k)$ then we have $\tilde \alpha_{\hat k_1 + \hat k_2} = \tilde \alpha_{\hat k_1} + \tilde \alpha_{\hat k_2}$ $\mu$-almost everywhere for each $\hat k_1, \hat k_2 \in \hat K$.  By hypothesis (c), $\hat K$ is at most countable, hence for $\mu$-almost every point $x \in X$, the map $x \mapsto \tilde \alpha_{\hat k}(x)$ is a homomorphism from $\hat K$ to $\T$, and hence by Pontryagin duality takes the form $\tilde \alpha_{\hat k}(x) = \langle \hat k, F(x) \rangle$ for some $\mu$-almost everywhere defined map $F \colon X \to K$, which one can verify to be measurable.  One can then check that
$$ \rho_\gamma = F \circ T^\gamma - F $$
$\mu$-almost everywhere, giving the claim.

\subsection{The uncountable Moore--Schmidt theorem}

The hypotheses (a), (b), (c) were used in the above proof, but one can ask if they are truly necessary for Theorem \ref{mst}.  Thus, we can ask whether the Moore--Schmidt theorem holds for actions of uncountable discrete groups $\Gamma$ on spaces $X$ that are not standard Lebesgue, with cocycles taking values in groups $K$ that are compact Hausdorff abelian, but not necessarily metrizable.  We refer to this setting as the ``uncountable'' setting for short, in contrast to the ``countable'' setting in which hypotheses such as (a), (b), (c) are imposed.  Our motivation for this is to remove similar regularity hypotheses from other results in ergodic theory, such as the Host-Kra structure theorem \cite{host2005nonconventional}, which rely at one point on the Moore--Schmidt theorem.  This in turn is motivated by the desire to apply such structure theory to such situations as actions of hyperfinite groups on spaces equipped with Loeb measure, which (as has been seen in such work as \cite{szegedy2012higher}, \cite{gtz}) is connected with the inverse conjecture for the Gowers norms in additive combinatorics.  We plan to address these applications in future work.

Unfortunately, a naive attempt to remove the hypotheses from Theorem \ref{mst} leads to counterexamples.  The main difficulty is the \emph{Nedoma pathology}: Once the compact Hausdorff abelian group $K$ is no longer assumed to be metrizable, the product Borel $\sigma$-algebra $\B(K) \otimes \B(K)$ can be strictly smaller than the Borel $\sigma$-algebra $\B(K \times K)$, and the group operation $+: K \times K \to K$, while still continuous, can fail to be measurable when $K \times K$ is equipped with the product $\sigma$-algebra $\B(K) \otimes \B(K)$: see Remark \ref{red-counter}.  As a consequence, one cannot even guarantee that the sum $f+g$ of two measurable functions $f,g \colon X \to K$ remains measurable, and so even the very definition of a $K$-valued measurable cocycle or coboundary becomes problematic if one insists on endowing $K$ with the Borel $\sigma$-algebra $\B(K)$.

Two further difficulties, of a more technical nature, also arise.  One is that if $X$ is no longer assumed to be standard Lebesgue, then tools such as disintegration may no longer be available; one similarly may lose access to measurable selection theorems when $K$ is not metrizable.  The other is that if $\Gamma$ is allowed to be uncountable or $K$ is allowed to be non-metrizable, then one may have to manipulate an uncountable number of assertions that each individually hold $\mu$-almost everywhere, but for which one cannot ensure that they \emph{simultaneously} hold $\mu$-almost everywhere, because the uncountable union of null sets need not be null.

To avoid these difficulties, we will make the following modifications to the setup of the Moore--Schmidt theorem, which turn out to be natural changes to make in the uncountable setting. 
The most important change, which is needed to avoid the Nedoma pathology, is to coarsen the $\sigma$-algebra on the compact group $K$, from the Borel $\sigma$-algebra to the Baire $\sigma$-algebra (see e.g.~\cite[Volume 2]{bogachev2006measure} for a reference):

\begin{definition}[Baire $\sigma$-algebra]\label{reduced-def}  If $K$ is a compact space, we define the \emph{Baire $\sigma$-algebra} $\Baire(K)$ to be the $\sigma$-algebra generated by all the continuous maps $f: K \to \R$. We use $K_\Baire$ to denote the concrete measurable space $K_\Baire = (K, \Baire(K))$.
\end{definition}

Since every closed subset $F$ of a compact metric space $S$ is the zero set of a real-valued continuous function $x \mapsto \mathrm{dist}(x,F)$, we see that the Baire $\sigma$-algebra $\Baire(K)$ of a compact space $K$ can equivalently be defined as the $\sigma$-algebra generated by all the continuous maps into compact metric spaces; another equivalent definition of $\Baire(K)$ is the $\sigma$-algebra generated by closed $G_\delta$ sets. Clearly $\Baire(K)$ is a subalgebra of $\B(K)$ which is equal to $\B(K)$ when $K$ is metrizable.  However, it can be strictly smaller; see Remark \ref{red-counter}.  In Proposition \ref{group-mes} we will show that if $K$ is a compact Hausdorff group, then the group operations on $K$ are measurable on $K_\Baire$, even if they need not be on $K$.  For this and other reasons, we view $K_\Baire$ as the ``correct'' measurable space structure to place on $K$ when $K$ is not assumed to be metrizable.  The observation that the Baire $\sigma$-algebra is generally better behaved than the Borel $\sigma$-algebra in uncountable settings is well known; see for instance \cite[\S 5.2]{EFHN}.

To avoid the need to rely on disintegration and measurable selection, and to avoid situations where we take uncountable unions of null sets, we shall  adopt a ``point-less'' or ``abstract'' approach to measure theory, by replacing concrete measurable spaces $(X,{\mathcal X})$ with their abstract counterparts.  Namely:

\begin{definition}[Abstract measurable spaces] The category $\AbsMes = \SigmaAlgop$ of abstract measurable spaces is the opposite\footnote{This is analogous to how the category of Stone spaces is equivalent to the opposite category of Boolean algebras, or how the category of affine schemes is equivalent to the opposite category of the category of commutative rings. One could also adopt a noncommutative probability viewpoint, and interpret  the category of abstract probability spaces as the opposite category to the category of tracial commutative von Neumann algebras, but we will not need to do so in this paper.} category of the category $\SigmaAlg$ of $\sigma$-complete Boolean algebras (or \emph{abstract $\sigma$-algebras}).  That is to say, an abstract measurable space (i.e., an object in $\AbsMes$) is a Boolean algebra $\X = (\X, 0, 1, \wedge, \vee, \overline{\cdot})$ is a Boolean algebra that is $\sigma$-complete (all countable families have meets and joins), and an abstract measurable map $f \in \Hom_{\AbsMes}(\X;\Y)$ (i.e., a morphism in $\AbsMes$) from one abstract measurable space $\X$ to another $\Y$ is a formal object of the form $f = (f^*)^{\mathrm{op}}$, where $f^*: \Y \to \X$ is a $\sigma$-complete homomorphism, that is to say a Boolean algebra homomorphism that also preserves countable joins: $f^* \bigvee_{n=1}^\infty E_n = \bigvee_{n=1}^\infty f^* E_n$ for all $E_n \in \Y$.  We refer to $f^*$ as the \emph{pullback map} associated to $f$.  Here $\op$ is a formal symbol to indicate use of the opposite category; the space $\Hom_{\AbsMes}(\X; \Y)$ is thus in one-to-one correspondence with the space $\Hom_{\SigmaAlg}(\Y; \X)$ of $\sigma$-complete Boolean homomorphisms from $\Y$ to $\X$.  If $f \in \Hom_\AbsMes( \X; \Y)$ and $g \in \Hom_\AbsMes( \Y; \mathcal{Z} )$ are abstract measurable maps, the composition $g \circ f \in \Hom_\AbsMes( \X; \mathcal{Z})$ is defined by the formula $g \circ f \coloneqq (f^* \circ g^*)^\op$ (or equivalently $(g \circ f)^* = f^* \circ g^*$).  Elements of the $\sigma$-complete Boolean algebra $\X$ will be also be referred to as \emph{abstract measurable subsets} of $\X$.

\end{definition}

We study the category of abstract measurable spaces in more detail in the followup paper \cite{jt-foundational}.

Note that any (concrete) measurable space $(X,\Sigma_X)$ can be viewed as an abstract measurable space by viewing the $\sigma$-algebra $\Sigma_X$ as a $\sigma$-complete Boolean algebra in the obvious manner (replacing set-theoretic symbols such as $\emptyset, X, \cup, \cap$ with their Boolean algebra counterparts $0, 1, \vee, \wedge$) and identifying $(X,\Sigma_X)$ (by some abuse of notation) with $\Sigma_X$, and similarly any (concrete) measurable map $f \colon X \to Y$ between two measurable spaces $(X,\Sigma_X), (Y,\Sigma_X)$ can be viewed as an abstract measurable map in $\Hom_\AbsMes(X; Y) = \Hom_\AbsMes(\Sigma_X; \Sigma_Y)$ by identifying $f$ with $(f^*)^\op$, where $f^* \colon \Sigma_Y \to \Sigma_X$ is the pullback map.  By abuse of notation, we shall frequently use these identifications in the sequel without further comment.  One can then easily check that the category $\ConcMes$ of concrete measurable spaces is a subcategory of the category $\AbsMes$ of abstract measurable spaces (in particular, the composition law for concrete measurable maps is consistent with that for abstract measurable maps).  

\begin{example}\label{ultra}  Let $\mathrm{pt}$ be a point (with the discrete $\sigma$-algebra); this is a concrete measurable space, which is identified with the abstract measurable space given by the $\sigma$-complete Boolean algebra $2^{\mathrm{pt}} = \{0,1\}$.  Then $\Hom_\AbsMes(\mathrm{pt}; \N)$ can be identified with $\N$ (with every natural number $n$ giving an abstractly measurable map $n \in \Hom_\AbsMes(\mathrm{pt}; \N) \equiv \Hom_{\SigmaAlg}(2^\N; \{0,1\})$ defined by $n^* E = 1_{n \in E}$ for $E \subset \N$).
\end{example}

An important further  example for us of an abstract measurable space (that is not, in general, represented by a concrete measurable space) will be as follows.  If $(X,\Sigma_X,\mu)$ is a measure space, we define the \emph{(opposite) measure algebra} $X_\mu$ to be the abstract measurable space $\Sigma_X/\mathcal{N}_\mu$, where $\mathcal{N}_\mu \coloneqq \{ A \in \X: \mu(A)=0\}$ is the $\sigma$-ideal of $\mu$-null sets, thus the abstract measurable subsets of $X_\mu$ are equivalence classes $[A] \coloneqq \{ A' \in \X: A \Delta A' \in \mathcal{N}_\mu \}$ for $A \in \X$.  We call $[A]$ the \emph{abstraction} of $A$ and $A$ a \emph{representative} of $[A]$.  
 
Informally, the measure algebra $X_\mu$ is formed from $X$ by ``removing the null sets'' (without losing any sets of positive measure); this is an operation that does not make sense on the level of concrete measurable spaces, but is perfectly well defined in the category of abstract measurable spaces.  The measure $\mu$ can be viewed as a countably additive map from the measure algebra $X_\mu$ to $[0,+\infty]$.  There is an obvious ``inclusion map'' $\iota \in \Hom_\AbsMes(X_\mu; X) \equiv \Hom_{\SigmaAlg}(\Sigma_X; \Sigma_X/\mathcal{N}_\mu)$, which is the abstract measurable map defined by setting $\iota^* A \coloneqq [A]$ for all $A \in \X$; this is a monomorphism in the category of abstract measurable spaces.  

If $f \colon X \to Y$ is a concrete measurable map, we refer to $[f] \coloneqq \iota \circ f \in \Hom_{\AbsMes}(X_\mu; Y)$ as the \emph{abstraction} of $f$, and $f$ as a \emph{realization} of $[f]$; chasing all the definitions, we see that $[f]^* E = [f^* E]$ for all measurable subsets $E$ of $Y$.  Note that if $f \colon X \to Y$, $g \colon X \to Y$ are concrete measurable maps that agree $\mu$-almost everywhere, then $[f] = [g]$.  The converse is only true in certain cases: see Section \ref{condrep-sec}.  Furthermore, there exist abstract measurable maps in $\Hom_{\AbsMes}(X_\mu; Y)$ that have no realizations as concrete measurable maps from $X$ to $Y$; again, see Section \ref{condrep-sec}.  As such, $\Hom_{\AbsMes}(X_\mu; Y)$ is not equivalent in general to the space $L^0(X; Y)$ of concrete measurable maps from $X$ to $Y$ up to almost everywhere equivalence, although the two spaces are still analogous in many ways.  Our philosophy is that $\Hom_{\AbsMes}(X_\mu; Y)$ is a superior replacement for $L^0(X; Y)$ in uncountable settings, as it exhibits fewer pathologies; for instance it behaves well with respect to arbitrary products, as seen in Proposition \ref{product}, whereas $L^0(X; Y)$ does not (see Example \ref{bad-ex}).  The main drawback of working with $X_\mu$ is the inability to use ``pointwise'' arguments; however, it turns out that most of the tools we really need for our applications can be formulated without reference to points.  (Here we follow the philosophy of ``conditional set theory'' as laid out in \cite{drapeau2016algebra}.)

\begin{example}\label{pointless} Let $X$ be the unit interval $[0,1]$ with the Borel $\sigma$-algebra and Lebesgue measure $\mu$.  Then $\Hom_{\AbsMes}(\mathrm{pt}; X_\mu)$ can be verified to be empty.  Thus $X_\mu$ contains no ``points'', which explains why one cannot use ``pointwise'' arguments when working with $X_\mu$ as a base space.  Note this argument also shows that $X_\mu$ is not isomorphic to a concrete measurable algebra.
\end{example}

Define $\Aut(X_\mu)$ to be the group of invertible elements $T = (T^*)^\op$ of $\Hom_{\AbsMes}(X_\mu; X_\mu)$. Any element of $\Aut(X,\X,\mu)$ can be abstracted to an element of $\Aut(X_\mu)$; in fact the abstraction lies in the subgroup $\Aut(X_\mu,\mu)$ of $\Aut(X_\mu)$ consisting of maps $T$ that also preserve the measure, $T_* \mu = \mu$, but we will not need this measure-preservation property in our formulation of the Moore--Schmidt theorem.  We also remark that there can exist elements of $\Aut(X_\mu,\mu)$ that are not realized\footnote{For a simple example, let $X = \{1,2,3\}$, let $\X$ be the $\sigma$-algebra generated by $\{1\}, \{2,3\}$, and let $\mu$ assign an equal measure of $1/2$ to $\{1\}$ and $\{2,3\}$.  Then there is an element of $\Aut(X_\mu,\mu)$ that interchanges the equivalence classes of $\{1\}$ and $\{2,3\}$, but it does not arise from any element of $\Aut(X,\X,\mu)$.  One can also modify Example \ref{nonrep} to generate further examples of non-realizable abstract measure-preserving maps; we leave the details to the interested reader.} by a concrete element of $\Aut(X,\X,\mu)$.  We believe that $\Aut(X_\mu)$ (or $\Aut(X_\mu,\mu)$) is a more natural replacement for $\Aut(X,\X,\mu)$ in the case when $X$ is not required to be standard Lebesgue.  An \emph{abstract action} of a discrete (and possibly uncountable) group $\Gamma$ on $X_\mu$ is defined to be a group homomorphism $\gamma \mapsto T^\gamma$ from $\Gamma$ to $\Aut(X_\mu)$.  Clearly any concrete measure-preserving action of $\Gamma$ on $X$ also gives rise to an abstract measure-preserving action on $X_\mu$, but there are abstract actions that are not represented by any concrete one (even if one is willing to work with ``near-actions'' in which the composition law $T^{\gamma_1} \circ T^{\gamma_2} = T^{\gamma_1 \gamma_2}$ only holds almost everywhere rather than everywhere).

If $(X,\X,\mu)$ is a probability space (not necessarily standard Lebesgue) and $K$ is a compact abelian group (not necessarily metrizable), then the measurable nature of the group operations on $K_\Baire$ makes the space $\Hom_{\AbsMes}(X_\mu; K_\Baire )$ an abelian group: see Section \ref{pont-sec}.  If $\Gamma$ is a (possibly uncountable) discrete group acting abstractly on $X_\mu$, we define an \emph{abstract $K$-valued cocycle} to be a collection $\rho = (\rho_\gamma)_{\gamma \in \Gamma}$ of abstract measurable maps $\rho_\gamma \in \Hom_{\AbsMes}(X_\mu; K_\Baire )$ such that
$$ \rho_{\gamma_1\gamma_2} = \rho_{\gamma_1} \circ T^{\gamma_2} + \rho_{\gamma_2}$$
for all $\gamma_1,\gamma_2 \in \Gamma$.  Note in comparison to \eqref{cocy} that we no longer need to introduce the caveat ``$\mu$-almost everywhere''.  We say that an abstract $K$-valued cocycle is an \emph{abstract coboundary} if there is an abstract measurable map $F \in \Hom_{\AbsMes}(X_\mu; K_\Baire )$ such that
$$ \rho_\gamma = F \circ T^\gamma - F $$
for all $\gamma \in \Gamma$.  

With these preliminaries, we are finally able to state the uncountable analogue of the Moore--Schmidt theorem. As a minor generalization, we can also allow $(X,\X,\mu)$ to be an arbitrary measure space rather than a probability space; in particular, $(X,\X,\mu)$ is no longer required to be $\sigma$-finite, again in the spirit of moving away from ``countably complicated'' settings.

\begin{theorem}[Uncountable Moore--Schmidt theorem]\label{mstu}
Let $\Gamma$ be a discrete group acting abstractly on the measure algebra $X_\mu$ (viewed as an abstract measurable space) of a measure space $X = (X,\X,\mu)$, and let $K$ be a compact Hausdorff abelian group.  Then an abstract $K$-valued cocycle $\rho = (\rho_\gamma)_{\gamma \in \Gamma}$ on $X_\mu$ is an abstract coboundary if and only if the $\T$-valued abstract cocycles $\hat k \circ \rho \coloneqq (\hat{k}\circ \rho_\gamma)_{\gamma \in \Gamma}$ are abstract coboundaries for all $\hat{k} \in \hat K$.
\end{theorem}

We prove this result in Section \ref{uncountable-sec}; the key tool is a ``conditional'' version of the Pontryagin duality relationship between $K$ and $\hat K$, which we formalize as Theorem \ref{cond-pont}.  Once this result is available, the proof mimics the proof of the countable Moore--Schmidt theorem, translated to the abstract setting.  We avoid the use of the ergodic decomposition by replacing the role of the scalars $\T$ by the invariant factor $\Hom_{\AbsMes}(X_\mu; \T)^\Gamma$.

While we believe that the formalism of abstract measure spaces is the most natural one for this theorem, one can still explore the question of to what extent Theorem \ref{mstu} continues to hold if one works with concrete actions, cocycles, and coboundaries instead of abstract ones.  We do not have a complete answer to this question, but we give some partial results in Sections \ref{condrep-sec}, \ref{concrete-sec}; in particular we recover Theorem \ref{mst} as a corollary of Theorem \ref{mstu}. 

\begin{remark} If $\Scal$ is an arbitrary abstract measurable space, then by the Loomis-Sikorski theorem \cite{loomis1947, sikorski2013boolean} $\Scal$ is isomorphic to $\X/{\mathcal N}$ for some concrete measurable space $(X,\X)$ and some null ideal ${\mathcal N}$ of $\X$.  In particular $\Scal$ is isomorphic to $X_\mu$, where $\mu$ is the (non-$\sigma$-finite) measure on $X$ that assigns $0$ to elements of ${\mathcal N}$ and $+\infty$ to all other elements.  Thus in Theorem \ref{mstu} one can replace the measure algebra $X_\mu$ by an arbitrary abstract measurable space.
\end{remark}

\subsection{Notation}

For any unexplained definition or result in the theory of measure algebras, we refer the interested reader to \cite{fremlin1989measure}, and for any unexplained definition or result in the general theory of Boolean algebras to \cite[Part 1]{monk1989handbook}. 

If $S$ is a statement, we use $1_S$ to denote its indicator, equal to $1$ when $S$ is true and $0$ when $S$ is false.  (In some cases, $1$ and $0$ will be interpreted as elements of a Boolean algebra, rather than as numbers.)

\subsection{Acknowledgments}
AJ was supported by DFG-research fellowship JA 2512/3-1. 
TT was supported by a Simons Investigator grant, the James and Carol Collins Chair, the Mathematical Analysis \& Application Research Fund Endowment, and by NSF grant DMS-1764034.

\section{The Baire $\sigma$-algebra}\label{reduced-sec}

In this section we explore some properties of the measurable spaces $K_\Baire = (K, \Baire(K))$  defined in Definition \ref{reduced-def}. We have already observed that $\Baire(K) = \B(K)$ when $K$ is a metric space.  The Baire $\sigma$-algebra also interacts well with products:

\begin{lemma}[Baire $\sigma$-algebras and products]\label{subspace-red}  Let $K$ be a closed subspace of a product $S_A \coloneqq \prod_{\alpha \in A} S_\alpha$ of compact spaces $S_\alpha$.  Then $\Baire(K)$ is the restriction of the product $\sigma$-algebra $\B_A \coloneqq \bigotimes_{\alpha \in A} \Baire(S_\alpha)$ to $K$:
$$ \Baire(K) = \{ E \cap K: E \in \B_A \}.$$
Equivalently, $\Baire(K)$ is the $\sigma$-algebra generated by the coordinate projections $\pi_\alpha \colon K \to (S_\alpha)_\Baire$, $\alpha \in A$.
\end{lemma}

We caution that this lemma does \emph{not} assert that $K$ itself lies in $\B_A$; see Remark \ref{red-counter} below for an explicit counterexample.  Also note that the index set $A$ is permitted to be uncountable.

\begin{proof}\footnote{We are indebted to the anonymous referee for this simplified proof.}  The collection of functions on $K$ of the form $f_\alpha \circ \pi_\alpha$ with $\alpha \in A$ and $f_\alpha \colon S_\alpha \to \R$ generate an algebra of continuous functions that separate points, hence by the Stone--Weierstrass theorem the $\sigma$-algebra they generate is equal to $\Baire(K)$.  The claim follows.
\end{proof}

Lemma \ref{subspace-red} combines well with the following theorem of Weil \cite{weil1937espaces}:

\begin{theorem}[Weil's theorem]\label{weil-thm} Every compact Hausdorff space is homeomorphic to a closed subset of a product of compact metric spaces.
\end{theorem}

Lemma \ref{subspace-red} also combines well with the following topological lemma:

\begin{lemma}\label{embed}  Let $K$ be a compact Hausdorff space, and let $\rho = (\rho_\alpha)_{\alpha \in A}$ be a family of continuous maps $\rho_\alpha \colon K \to S_\alpha$ from $K$ to compact Hausdorff spaces $S_\alpha$.  Suppose that the $\rho_\alpha$ separate points, thus for every distinct $k,k' \in K$ there exists $\alpha \in A$ such that $\rho_\alpha(k) \neq \rho_\alpha(k')$.  We view $\rho \colon K \to S_A$ as a map from $K$ to $S_A$ by setting $\rho(k) \coloneqq (\rho_\alpha(k))_{\alpha \in A}$.  Then $\rho(K)$ is a closed subset of $S_A$, and $\rho$ is a homeomorphism between $K$ and $\rho(K)$ (where we give the latter the topology induced from the product topology on $S_A$).
\end{lemma}

\begin{proof}  Clearly $\rho$ is continuous and injective (since the $\rho_\alpha$ separate points), so $\rho(K)$ is compact and hence closed in the Hausdorff space $S_A$.  Thus $\rho \colon K \to \rho(K)$ is a continuous bijection between compact Hausdorff spaces; it therefore maps compact sets to compact sets, hence is an open map, hence is a homeomorphism as required. 
\end{proof}

In the case when $K$ is a group, we can give a more explicit description of an embedding $\rho$ of the form described in Lemma \ref{embed}:

\begin{corollary}[Description of compact Hausdorff groups]\label{group-desc}  Let $K$ be a compact Hausdorff group.
\begin{itemize}
    \item[(i)]  There exists a family $\rho = (\rho_\alpha)_{\alpha \in A}$ of continuous unitary representations $\rho_\alpha \colon K \to S_\alpha$, $\alpha \in A$, of $K$ (thus each $S_\alpha$ is a unitary group and $\rho_\alpha$ is a continuous homomorphism) such that $\rho(K)$ is a closed subgroup of $S_A$, and $\rho \colon K \to \rho(K)$ is an isomorphism of topological groups.  The $\sigma$-algebra $\Baire(K)$ is generated by the representations $\rho_\alpha$.
    \item[(ii)]  If $K = (K,+)$ is abelian, and one defines the map $\iota \colon K \to \T^{\hat K}$ by $\iota(k) \coloneqq (\langle \hat k, k \rangle)_{\hat k \in \hat K}$, then $\iota(K)$ is a closed subgroup of $\T^{\hat K}$, and $\iota \colon K \to \iota(K)$ is an isomorphism of topological groups.  The $\sigma$-algebra $\Baire(K)$ is generated by the characters $\hat k \in \hat K$.  Furthermore, one can describe $\iota(K)$ explicitly as
    \begin{equation}\label{iotak} \iota(K) = \{ (\theta_{\hat k})_{\hat k \in \hat K} \in \T^{\hat K}: \theta_{\hat k_1 + \hat k_2} = \theta_{\hat k_1} + \theta_{\hat k_2} \forall \hat k_1, \hat k_2 \in \hat K \}.
    \end{equation}
\end{itemize}
\end{corollary}

\begin{proof} For part (i), we observe from the Peter-Weyl theorem that there are enough continuous unitary representations of $K$ to separate points, and the claim now follows from Lemma \ref{embed} and Lemma \ref{subspace-red}.

For part (ii), we observe from Plancherel's theorem that the characters $\hat k \colon K \to \T$ for $\hat k \in \hat K$ separate points, so by Lemma \ref{embed} we verify that $\iota(K)$ is a closed subgroup of $\T^{\hat K}$ and that $\iota \colon K \to \iota(K)$ is an isomorphism of topological groups, and from Lemma \ref{subspace-red} we see that $\Baire(K)$ is generated by the characters $\hat k \in \hat K$.  As $K$ is compact, the Pontryagin dual $\hat K$ is discrete, and by Pontryagin duality, $K$ can be identified with the space of homomorphisms $\hat k \mapsto \theta_{\hat k}$ from $\hat K$ to $\T$.  This gives the description \eqref{iotak}.
\end{proof}

As a consequence of Corollary \ref{group-desc}, we have 

\begin{proposition}[Group operations measurable in Baire $\sigma$-algebra]\label{group-mes}  Let $K = (K,\cdot)$ be a compact Hausdorff group.  Then the group operations $\cdot: K_\Baire \times K_\Baire \to K_\Baire$ and $()^{-1}: K_\Baire \to K_\Baire$ are measurable.  In particular, if $K = (K,+)$ is a compact Hausdorff abelian group, then the group operations $+: K_\Baire \times K_\Baire \to K_\Baire$ and $-: K_\Baire \to K_\Baire$ are measurable.
\end{proposition}

\begin{proof}  By Corollary \ref{group-desc}(i), we may view $K_\Baire$ as a closed subgroup of a product of unitary groups.  The group operations are measurable on each such unitary group, hence measurable on the product, giving the claim.
\end{proof}

\begin{remark}[Nedoma pathology]\label{red-counter}  Let $K$ be the non-metrizable compact Hausdorff abelian group $K = \T^\R$, and let $K^\Delta \subset K \times K$ be the diagonal closed subgroup $K^\Delta = \{ (k,k): k \in K \}$.  By \emph{Nedoma's pathology}  \cite{nedoma1957note}, $K^\Delta$ is not measurable in $\B(K) \otimes \B(K)$.  Indeed, $\B(K) \otimes \B(K)$ consists of the union of $\B_1 \otimes \B_2$ as $\B_1, \B_2$ range over countably generated subalgebras of $\B(K)$.  If $K^\Delta$ were in $\B(K) \otimes \B(K)$, we conclude on taking slices that all the points in $K$ lie in a single countably generated subalgebra of $\B(K)$, but the latter has cardinality at most $2^{\aleph_0}$ and the former has cardinality $2^{2^{\aleph_0}}$, leading to a contradiction.  This shows that $\B(K) \otimes \B(K) \neq \B(K \times K)$, and also shows that in Lemma \ref{subspace-red} $K$ need not be measurable in $S_A$.  Also, by comparing this situation with Proposition \ref{group-mes}, we conclude that $\B(K) \neq \Baire(K)$ in this case.  This can also be seen directly: $\Baire(K)$ is the product $\sigma$-algebra on $\T^\R$, which is also equal to the union of the pullbacks of the $\sigma$-algebras of $\T^I$ for all countable subsets of $I$.  In particular a single point in $K$ will not be measurable in $\Baire(K)$, even though it is clearly measurable in $\B(K)$.
\end{remark}

\section{A conditional Pontryagin duality theorem}\label{pont-sec}

Throughout this section, $X = (X,\Sigma_X,\mu)$ denotes a measure space; to avoid some degeneracies we will assume in this section that $X$ has positive measure.  We will use the abstract measurable space $X_\mu$ as a base space for the formalism of conditional set theory and conditional analysis, as laid out in \cite{drapeau2016algebra} (although as it turns out we will not need to draw upon the full power\footnote{For instance, we will not utilize the (measurable) topos-theoretic ability, which is powered by the completeness of $X_\mu$ when viewed as a Boolean algebra (which is equivalent to $X_\mu$ being decomposable, and in particular is the case if $(X,\Sigma_X,\mu)$ is $\sigma$-finite, but is an assumption we will not need in our analysis), to glue together different conditional objects along a partition of the base space $X_\mu$, which allows one to develop in particular a theory of conditional metric spaces and conditional topology.} of this theory in this paper). In this formalism, many familiar objects such as numbers, sets, and functions will have ``conditional'' analogues which vary ``measurably'' with the base space $X_\mu$; to avoid confusion, we will then use the term ``classical'' to refer to the original versions of these concepts.  Thus for instance we will have classical real numbers and conditional real numbers, classical functions and conditional functions, and so forth.  The adjectives ``classical'' and ``conditional'' in this formalism are analogous to the adjectives ``deterministic'' and ``random'' in probability theory (for instance the latter theory deals with both deterministic real numbers and random real variables).  Our ultimate objective of this section is to obtain a conditional analogue of the Pontryagin duality identity \eqref{iotak}.

We begin with some basic definitions.

\begin{definition}[Conditional spaces]  If $Y = (Y,\Y)$ is any concrete measurable space, we define the \emph{conditional analogue}
$\Cond(Y) = \Cond_{X_\mu}(Y)$ of $Y$ to be the space $\Cond(Y) \coloneqq \Hom_{\AbsMes}(X_\mu; Y)$.  Elements of $\Cond(Y)$ will be referred to as \emph{conditional elements} of $Y$.  Thus for instance elements of $\Cond(\R) = \Hom_{\AbsMes}(X_\mu; \R)$ are conditional reals, and elements of $\Cond(\N) = \Hom_{\AbsMes}(X_\mu; \N)$ are conditional natural numbers.  Every (classical) element $y \in Y$ gives rise to a constant abstract measurable map $\Cond(y) \in \Cond(Y)$, defined by setting $\Cond(y)^* A = 1_{y \in A}$ for $A \in \Y$ (where the indicator $1_{y \in A}$ is interpreted as taking values in the $\sigma$-complete Boolean algebra $\X_\mu$). We will usually abuse notation by referring\footnote{This is analogous to how a constant function $x \mapsto c$ that takes a fixed value $c \in Y$ for all inputs $x \in X$ is often referred to (by abuse of notation) as $c$. Strictly speaking, in order for the identification of $y$ with $\Cond(y)$ to be injective, $\Y$ needs to separate points (i.e., for any distinct $y,y'$ in $Y$ there exists $A \in \Y$ that contains $y$ but not $y'$), but we will ignore this subtlety when abusing notation in this manner.} to $\Cond(y)$ simply as $y$.

\end{definition}

Thus for instance if $\rho = (\rho_\gamma)_{\gamma \in \Gamma}$ is an abstract $K$-valued cocycle, then each $\rho_\gamma$ is a conditional element of $K_\Baire$.

As discussed in the introduction, every concrete measurable map $f: X \to Y$ into a concrete measurable space $Y$ gives rise to a conditional element $[f] \in \Cond(Y)$.  In the case that $X$ is a Polish space, this is an equivalence:

\begin{proposition}[Conditional elements of compact metric or Polish spaces]\label{metr}  Let $K$ be a Polish space.  Then every conditional element $k \in \Cond(K)$ has a realization by a concrete measurable map $F \colon X \to K$, unique up to $\mu$-almost everywhere equivalence.
\end{proposition}


\begin{proof}  Since $X$ has positive measure, $X_\mu$ is non-trivial, and hence we may assume $K$ is non-empty (since otherwise there are no conditional elements of $K$).

First suppose that $K$ is Polish.  We may endow $K$ with a complete metric $d$. The space $K$ is separable, and hence for every $n \in \N$ there exists a measurable ``rounding map'' $f_n: K \to S_n$ to an at most countable  subset $S_n$ of $K$ with the property that 
\begin{equation}\label{fnk}
d(k', f_n(k')) \leq \frac{1}{n}
\end{equation}
for all $k' \in K$.  If $k \in \Cond(K) = \Hom_\AbsMes(X_\mu;K)$, then $f_n \circ k \in \Cond(S_n) = \Hom_\AbsMes(X_\mu; S_n)$ (since $f_n$ can be viewed as an element of $\Hom_\AbsMes(K; S_n)$).  By taking representatives of the preimages $(f_n \circ k)^* \{s\}  = k^*(f_n^*(\{s\}))$ for each $s \in S_n$, and adjusting these representatives by null sets to form a partition of $X$, we can find a measurable realization $F_n \colon X \to S_n$ of $f_n \circ k$.  Since $d(f_n(k'),f_m(k')) \leq \frac{1}{n}+\frac{1}{m}$ for all $n,m \in \N$ and $k' \in K$, we have $d(F_n(x), F_m(x)) \leq \frac{1}{n} + \frac{1}{m}$ for each $n,m \in \N$ and $\mu$-almost every $x \in X$.  Thus the sequence of measurable functions $F_n\colon X \to K$ is almost everywhere Cauchy, and thus (see e.g., \cite[Lemmas 1.10, 4.6]{kallenberg2002}) converges $\mu$-almost everywhere  to a measurable limit $F\colon X \to K$.  To finish the claim of existence, it suffices to show that $[F] = k$, that is to say that
$$ [F^*(E)] = k^*(E)$$
for all Borel subsets $E$ of $K$.  Since this claim is preserved under $\sigma$-algebra operations, we may assume without loss of generality that $E$ is an open ball $E = B(k_0,r)$.  Let $0 < r_1 < r_2 < \dots < r$ be a strictly increasing sequence of radii converging to $r$.  If $m > 2$, then since the $F_n$ converge almost everywhere to $F$, we have
$$ \limsup_{n \to \infty} [F_n^*(B(k_0, r_{m-1}))]
\leq [F^*(B(k_0, r_m))] \leq \liminf_{n \to \infty} [F_n^*(B(k_0, r_{m+1}))]
$$
in the $\sigma$-complete Boolean algebra $X_\mu$.  But when $n$ is sufficiently large depending on $m$, we have from \eqref{fnk} that
$$ [F_n^*(B(k_0, r_{m-1}))] = k^*( f_n^*(B(k_0, r_{m-1})) ) \geq k^*(B(k_0,r_{m-2}))$$
and
$$ [F_n^*(B(k_0, r_{m+1}))] = k^*( f_n^*(B(k_0, r_{m+1})) ) \leq k^*(B(k_0,r_{m+2}))$$
and thus we have
$$ k^*(B(k_0, r_{m-2}))
\leq [F^*(B(k_0, r_m))] \leq k^*(B(k_0, r_{m+2}))
$$
for all $m>2$.  Sending $m \to \infty$, using the $\sigma$-complete homomorphism nature of both $k^*$ and $F^*$, we conclude that 
$$ [F^*(B(k_0,r))] = k^*(B(k_0,r))$$
as required.

For uniqueness, suppose that $F,G\colon X \to K$ are two measurable maps with $[F]=[G]$, thus $F^* E$ differs by a null set from $G^* E$ for every measurable $E \in K$.  If $F$ is not equal almost everywhere to $G$, then $d(F,G) > 0$ on a set of positive measure, and then by the second countable nature of $K$ we may find a ball $B$ for which $F^* B$ and $G^* B$ differ by a set of positive measure, a contradiction.  Thus $F$ is equal to $G$ $\mu$-almost everywhere as claimed.
\end{proof}

Now we look at conditional elements of arbitrary products $\prod_{\alpha \in A} S_\alpha = (\prod_{\alpha \in A} S_\alpha, \bigotimes_{\alpha \in A} \Scal_\alpha)$ of Polish spaces $S_\alpha = (S_\alpha, \Scal_\alpha)$.  Here, as is usual, $\prod_{\alpha \in A} S_\alpha$ is the Cartesian product, and the product $\sigma$-algebra $\bigotimes_{\alpha \in A} \Scal_\alpha$ is the minimal $\sigma$-algebra that makes all the projection maps $\pi_\beta \colon \prod_{\alpha \in A} S_\alpha \to S_\beta$ measurable for $\beta \in A$.   We have the following fundamentally important identity:

\begin{proposition}[Conditional elements of product spaces]\label{product}  Let $(S_\alpha)_{\alpha \in A}$ be a family of Polish spaces $S_\alpha = (S_\alpha,\Scal_\alpha)$.  Then one has the equality
$$ \Cond\left(\prod_{\alpha \in A} S_\alpha\right) = \prod_{\alpha \in A} \Cond(S_\alpha)$$
formed by identifying each conditional element $f$ of $\prod_{\alpha \in A} S_\alpha$ with the tuple $(\pi_\alpha \circ f)_{\alpha \in A}$.
\end{proposition}

\begin{proof}  It is clear that if $f \in  \Cond(\prod_{\alpha \in A} S_\alpha)$ then $(\pi_\alpha \circ f)_{\alpha \in A}$ lies in $\prod_{\alpha \in A} \Cond(S_\alpha)$.  Now  suppose that $(f_\alpha)_{\alpha \in A}$ is an element of $\prod_{\alpha \in A} \Cond(S_\alpha)$.  By Proposition \ref{metr}, for each $\alpha \in A$ we can find a concrete measurable map $\tilde f_\alpha: X \to S_\alpha$ such that $f_\alpha = [\tilde f_\alpha]$.  Let $\tilde f: X \to \prod_{\alpha \in A} S_\alpha$ be the map
$$ \tilde f(x) \coloneqq (\tilde f_\alpha(x))_{\alpha \in A},$$
then $\tilde f$ is a concrete measurable map.  Set $f \coloneqq [\tilde f]$, then $f \in \Cond\left(\prod_{\alpha \in A} S_\alpha\right)$.  By chasing all the definitions we see that $(\pi_\alpha \circ f)^* E = f_\alpha^* E$ for any $E \in \Scal_\alpha$, hence $(f_\alpha)_{\alpha \in A} = (\pi_\alpha \circ f)_{\alpha \in A}$.

It remains to show that each tuple $(f_\alpha)_{\alpha \in A}$ is associated to at most one $f \in \Cond(\prod_{\alpha \in A} S_\alpha)$.  Suppose that $f,g \in \Cond(\prod_{\alpha \in A} S_\alpha)$ are such that $\pi_\alpha \circ f = \pi_\alpha \circ g$ for all $\alpha \in A$.  Then we have $f^* E = g^* E$ for all generating elements $E$ of the product $\sigma$-algebra $\bigotimes_{\alpha \in A} \Scal_\alpha$.  As $f^*, g^*$ are both $\sigma$-algebra homomorphisms, we conclude that $f^* = g^*$ and hence $f=g$, giving the claim.
\end{proof}

The hypothesis that $S_\alpha$ are Polish cannot be relaxed to arbitrary concrete measurable spaces, even when considering products of just two spaces; see Proposition \ref{counter-prop}.

If $f \colon Y \to Z$ is a (classical) concrete measurable map between two concrete measurable spaces $Y,Z$, then we can define the conditional analogue $\Cond(f) \colon \Cond(Y) \to \Cond(Z)$ of this function by the formula
$$ \Cond(f)(y) \coloneqq f \circ y$$
for $y \in \Cond(Y)$. 
By chasing the definitions, we also observe the functoriality property
\begin{equation}\label{functor}
    \Cond(g \circ f) = \Cond(g) \circ \Cond(f)
\end{equation}
whenever $f \colon Y \to Z$, $g \colon Z \to W$ are classical concrete measurable maps between concrete measurable spaces $Y,Z,W$; using the identification from Proposition \ref{product} we also have the identity
\begin{equation}\label{condpair}
    (\Cond(f_1), \Cond(f_2)) = \Cond((f_1,f_2))
\end{equation}
for any classical concrete measurable maps $f_1 \colon K \to S_1$, $f_2 \colon K \to S_2$ from a measurable space $K$ to Polish spaces $S_1, S_2$, and more generally
\begin{equation}\label{cond-tuple}
    (\Cond(f_\alpha))_{\alpha \in A} = \Cond((f_\alpha)_{\alpha \in A})
\end{equation}
whenever $f_\alpha \colon K \to S_\alpha$, $\alpha \in A$, are classical concrete measurable maps from a measurable space $K$ to Polish spaces $S_\alpha$.

Suppose that $S$ is a concrete measurable space and $K$ is a (possibly non-measurable) subset of $S$, then the measurable space structure on $S$ induces one on $K$ by restricting all the measurable sets of $S$ to $K$.  The inclusion map $\iota: K \to S$ is then measurable, and thus $\Cond(\iota)$ is a conditional map from $\Cond(K)$ to $\Cond(S)$, which is easily seen to be injective; thus (by abuse of notation) we can view $\Cond(K)$ as a subset of $\Cond(S)$.  One can then ask for a description of this subset.  We can answer this in two cases:

\begin{proposition}[Description of $\Cond(K)$]\label{condk-desc}  Let $S = (S, \mathcal{S})$ be a concrete measurable space, let $K$ be a subset of $S$ with the induced measurable space structure $(K,\mathcal{K})$, and view $\Cond(K)$ as a subset of $\Cond(S)$ as indicated above.
\begin{itemize}
    \item[(i)] If $K$ is measurable in $S$, then $\Cond(K)$ consists of those conditional elements $s \in \Cond(S)$ of $S$ such that $s^* K=1$.
    \item[(ii)] If $S = S_A = \prod_{\alpha \in A} S_\alpha$ is the product of compact metric spaces $S_\alpha$ with the product $\sigma$-algebra, and $K$ is a closed (but not necessarily measurable) subset of $S_A$, then $\Cond(K)$ consists of those conditional elements $s_A \in \Cond(S_A)$ of $S_A$ such that $s_A^* \pi_I^{-1}(\pi_I(K))  = 1$ for all at most countable $I \subset A$, where $\pi_I: S_A \to S_I$ is the projection to the product $S_I \coloneqq \prod_{i \in I} S_i$.
\end{itemize}
\end{proposition}

\begin{proof}
For part (i), it is clear that if $k \in \Cond(K)$ then $k^* K=1$.  Conversely, if $s^* K=1$, then $s^* K^c=0$, and hence $s^* E = s^* F$ whenever $E,F$ are measurable subsets of $S$ that agree on $K$ (since $s^*(E \cap K^c) = s^*(F \cap K^c) = 0$).  Thus the $\sigma$-complete Boolean homomorphism $s^*: \mathcal{S} \to \X_\mu$ descends to a $\sigma$-complete Boolean homomorphism on $\mathcal{K}$, so that $s \in \Cond(K)$ as claimed.

Now we prove part (ii).  If $k \in \Cond(K)$ and $I \subset A$ is at most countable, then the image $\pi_I(K)$ is a compact subset of the metrizable space $S_I$, and is hence measurable in $S_I$; this also implies that $\pi_I^{-1}(\pi_I(K))$ is measurable in $S_A$.  Observe that $\Cond(\pi_I)(k)$ is an element of $\Cond(\pi_I(K))$, hence by (i) we have $\Cond(\pi_I)(k)^* \pi_I(K) = 1$, and hence $k^*( \pi_I^{-1}(\pi_I(K))) = 1$.

Conversely, assume that $s_A \in \Cond(S_A)$ is such that $s_A^*  \pi_I^{-1}(\pi_I(K))  = 1$ for all at most countable $I \subset A$. Let $E$ be a measurable subset of $S_A$ that was disjoint from $K$.  The product $\sigma$-algebra $\bigotimes_{\alpha \in A} \B(S_\alpha)$ is equal to the union of the pullbacks $\pi_I^*(\bigotimes_{i \in I} \B(S_i))$ as $I$ ranges over countable subsets of $A$ (since the latter is a $\sigma$-algebra contained in the former that contains all the generating sets).  Thus there exists an at most countable $I$ such that $E = \pi_I^{-1}(E_I)$ for some measurable subset $E_I$ of $S_I$.  Since $E$ is disjoint from $K$, $E_I$ is disjoint from $\pi_I(K)$, hence $E$ is disjoint from $\pi_I^{-1}(\pi_I(K))$.  Since $s_A^* \pi_I^{-1}(\pi_I(K))=1$, we conclude that $s_A^* E=0$ for all measurable $E$ disjoint from $K$.  Thus $s_A^* E = s_A^* F$
 whenever $E,F$ are measurable subsets of $S_A$ that agree on $K$, and by arguing as in (i) we conclude that $s \in \Cond(K)$, giving (ii).
\end{proof}

As a corollary we have the following variant of Proposition \ref{product}:

\begin{corollary}[Conditional elements of product spaces, II]\label{square}  Let $K, K'$ be compact Hausdorff spaces.  Then $\Cond(K_\Baire \times K'_\Baire) = \Cond(K_\Baire) \times \Cond(K'_\Baire)$.
\end{corollary}

The proof given below extends (with only minor notational changes) to arbitrary products of compact Hausdorff spaces, not just to products of two spaces, but the latter case is the only one we need in this paper.  We also give a generalization of Corollary \ref{square} in Proposition \ref{half-square}, in the case that $X$ is a probability space.

\begin{proof}  By Theorem \ref{weil-thm} and Lemma \ref{subspace-red}, we may assume $K_\Baire$ is a subspace of a product $S_A = \prod_{\alpha \in A} S_\alpha$ of compact metric spaces $S_\alpha$, with the $\sigma$-algebra induced from the product $\sigma$-algebra, and similarly that $K'_\Baire$ is a subspace of $S'_{A'} = \prod_{\alpha \in A'} S'_\alpha$.  From Proposition \ref{condk-desc}(ii), $\Cond(K_\Baire)$ consists of those elements $s_A \in \Cond(S_A)$ such that $s_A^* \pi_I^{-1}(\pi_I(K))  = 1$ for all at most countable $I \subset A$.  Similarly for $\Cond(K'_\Baire)$.  From Lemma \ref{subspace-red} we have $K_\Baire \times K'_\Baire = (K \times K')_\Baire$, and from Proposition \ref{product} we have $\Cond(S_A \times S'_{A'}) = \Cond(S_A) \times \Cond(S'_{A'})$, so by a second application of Proposition \ref{condk-desc} we see that $\Cond(K_\Baire \times K'_\Baire)$ consists of those elements $(s_A, s'_{A'}) \in \Cond(S_A) \times \Cond(S'_{A'})$ such that
$$ (s_A, s'_{A'})^* ( \pi_I^{-1}(\pi_I(K)) \times \pi_{I'}^{-1}(\pi_{I'}(K')) ) = s_A^* \pi_I^{-1}(\pi_I(K)) \wedge (s'_{A'})^* \pi_{I'}^{-1}(\pi_{I'}(K')) = 1$$
for all at most countable $I \subset A, I' \subset A'$.  The claim follows.
\end{proof}

We can use conditional analogues of classical functions to generate various operations on conditional elements of concrete measurable spaces.  For instance, suppose we have two conditional real numbers $x,y \in \Cond(\R)$.  Then we can define their sum $x+y \in \Cond(\R)$ by the formula
\begin{equation}\label{plus-def}
 x+y = \Cond(+)(x,y)
 \end{equation}
where we use Proposition \ref{product} to view $(x,y)$ as an element of $\Cond(\R^2)$, and $+: \Cond(\R^2) \to \Cond(\R)$ is the conditional analogue of the classical addition map $+: \R^2 \to \R$.  Similarly for the other arithmetic operations; one then easily verifies using \eqref{functor}, \eqref{condpair} that the space $\Cond(\R)$ of conditional real numbers has the structure of a real unital commutative algebra.  This is analogous to the more familiar fact that $L^0(X; \R)$ is also a real unital commutative algebra.  A similar argument (using Proposition \ref{group-mes} and Corollary \ref{square}) shows that if $K$ is a compact Hausdorff group then $\Cond(K_\Baire)$ is also a group, which will be abelian if $K$ is abelian, and the group operations are conditional functions in the sense given in \cite{drapeau2016algebra}.

Now we can give a conditional analogue of the Pontryagin duality relationship \eqref{iotak}.

\begin{theorem}[Conditional Pontryagin duality]\label{cond-pont}  Let $K$ be a compact Hausdorff abelian group, and let $\iota \colon K_\Baire \to \T^{\hat K}$ be the map
$$ \iota(k) \coloneqq ( \langle \hat k, k \rangle )_{\hat k \in \hat K}.$$
Then 
\begin{equation}\label{iotak-cond} \Cond(\iota)(\Cond(K_\Baire)) = \{ (\theta_{\hat k})_{\hat k \in \hat K} \in \Cond(\T)^{\hat K}: \theta_{\hat k_1 + \hat k_2} = \theta_{\hat k_1} + \theta_{\hat k_2} \forall \hat k_1, \hat k_2 \in \hat K \}
\end{equation}
where we use Proposition \ref{product} to identify $\Cond(\T^{\hat K})$ with $\Cond(\T)^{\hat K}$.  Also, $\Cond(\iota): \Cond(K_\Baire) \to \Cond(\T^{\hat K})$ is injective.
\end{theorem}

\begin{proof}  For all $\hat k_1, \hat k_2 \in \hat K$, we have from definition of the group structure on $\hat K$ that
$$ \langle \hat k_1 + \hat k_2, k \rangle = \langle \hat k_1, k \rangle + \langle \hat k_2, k \rangle$$
for all classical elements $k \in K_\Baire$.  All expressions here are measurable in $k$, so the identity also holds for conditional elements $k \in \Cond(K_\Baire)$ (where by abuse of notation we write $\Cond(\langle \hat k, \cdot \rangle)$ simply as $\langle \hat k, \cdot \rangle$ for any $\hat k \in \hat K$).  From this we see that if $k \in \Cond(K_\Baire)$ then $\Cond(\iota)(k)$ lies in the set in the right-hand side of \eqref{iotak-cond}.

Now we establish the converse inclusion.  By Corollary \ref{group-desc}(ii), $\iota$ is a measurable space isomorphism between $K_\Baire$ and $\iota(K)$ (where the latter is given the measurable space structure induced from $\T^{\hat K}$). Thus $\Cond(\iota)$ is injective and $\Cond(\iota)(\Cond(K_\Baire)) = \Cond(\iota(K))$.  Let $\theta = (\theta_{\hat k})_{\hat k \in \hat K}$ be an element of the right-hand side of \eqref{iotak-cond}; we need to show that $\theta \in \Cond(\iota(K))$.  By Proposition \ref{condk-desc}(ii), it suffices to show that $\theta^* \pi_I^{-1}(\pi_I(\iota(K))) = 1$ for all at most countable $I \subset \hat K$.  By replacing $I$ with the group generated by $I$, which is still at most countable, it suffices to do so in the case when $I$ is an at most countable subgroup of $\hat K$.

Let $K_I \subset \T^I$ denote the group of homomorphisms from $I$ to $\T$, thus
$$ K_I = \{ (\xi_i)_{i \in I} \in \T^I: \xi_{i+j} = \xi_i + \xi_j \forall i,j \in I \}.$$
This is a closed subgroup of $\T^I$.  Because $\T$ is a divisible abelian group, we see from Zorn's lemma that every homomorphism from $I$ to $\T$ can be extended to a homomorphism from $\hat K$ to $\T$, thus $K_I = \pi_I(\iota(K))$.  From the hypotheses on $\theta$ we see that $(\theta_i)_{i \in I}$ is a conditional element of $K_I$, which by Proposition \ref{condk-desc}(i) implies that $(\theta_i)_{i \in I}^*  K_I=1$, and hence 
$$ \theta^* \pi_I^{-1}(\pi_I(\iota(K))) = \theta^* \pi_I^{-1}(K_I) = (\theta_i)_{i \in I}^* K_I = 1$$
giving the claim.
\end{proof}

\section{Proof of the uncountable Moore--Schmidt theorem}\label{uncountable-sec}

We now have enough tools to prove Theorem \ref{mstu}, by modifying the argument sketched in the introduction to prove Theorem \ref{mst}.  We may assume that the space $X$ has positive measure, since if $X$ has zero measure then every abstract cocycle is trivially an abstract coboundary.

Let $\Gamma$ be a discrete group acting abstractly on the measure algebra $X_\mu$ of an arbitrary measure space, and let $K$ be a compact Hausdorff abelian group.  If $\rho = (\rho_\gamma)_{\gamma \in \Gamma}$ is an abstract $K$-valued coboundary, then by definition there exists $F \in \Cond( K_\Baire )$ such that
$$ \rho_\gamma = F \circ T^\gamma - F $$
for all $\gamma \in \Gamma$, hence for each $\hat k \in \hat K$ we have
$$ \langle \hat k, \rho_\gamma\rangle = \langle \hat k, F\rangle \circ T^\gamma - \langle \hat k, F\rangle$$
for all $\gamma \in K$.  Thus each $\langle \hat k, \rho \rangle$ is an abstract $\T$-valued coboundary.

Conversely, suppose that for each $\hat k \in \hat K$, $\langle \hat k, \rho \rangle$ is an abstract $\T$-valued coboundary; thus we may find $\alpha_{\hat k} \in \Cond(\T)$ such that
\begin{equation}\label{hkr}
 \langle \hat k, \rho_\gamma\rangle = \alpha_{\hat k} \circ T^\gamma - \alpha_{\hat k}
 \end{equation}
 for all $\hat k \in \hat K$ and $\gamma \in \Gamma$.  If $\hat k_1, \hat k_2 \in \hat K$, then we have
$$ \langle \hat k_1 + \hat k_2, \rho_\gamma \rangle = \langle \hat k_1, \rho_\gamma \rangle + \langle \hat k_2, \rho_\gamma \rangle$$
which when combined with \eqref{hkr} and rearranging gives
$$ c(\hat k_1, \hat k_2) \circ T^\gamma = c(\hat k_1, \hat k_2)$$
where $c(\hat k_1, \hat k_2) \in \Cond(\T)$ is the conditional torus element
\begin{equation}\label{ckk}
c(\hat k_1, \hat k_2) \coloneqq \alpha_{\hat k_1 + \hat k_2} - \alpha_{\hat k_1} - \alpha_{\hat k_2}.
\end{equation}
Thus, if we define the invariant subgroup
$$ \Cond(\T)^\Gamma \coloneqq \{ \theta \in \Cond(\T): \theta \circ T^\gamma = \theta \; \forall \gamma \in \Gamma \}$$
of $\Cond(\T)$, then we have $c(\hat k_1, \hat k_2) \in \Cond(\T)^\Gamma$ for all $\hat k_1, \hat k_2 \in \hat K$.

We now claim that $\Cond(\T)^\Gamma$ is a divisible abelian group; thus for any $\theta \in \Cond(\T)^\Gamma$ and $n \in \N$, we claim that there exists $\beta \in \Cond(\T)^\Gamma$ such that $n\beta = \theta$. But one can easily construct a concrete measurable map $g_n \colon \T \to \T$ such that $n g_n(\theta) = \theta$ for all $\theta \in \T$ (for instance, one can set $g_n(x \mod \Z) \coloneqq \frac{x}{n} \mod \Z$ for $0 \leq x < 1$), and the claim then follows by setting $\beta \coloneqq \Cond(g_n)(\theta)$.

Since $\Cond(\T)^\Gamma$ is a divisible abelian subgroup of $\Cond(\T)$, we see from Zorn's lemma that there exists a retract homomorphism $w: \Cond(\T) \to \Cond(\T)^\Gamma$ (a homomorphism that is the identity on $\Cond(\T)^\Gamma$); see e.g.  \cite[p.~46--47]{halmos2013lectures}.  For each $\hat k \in \hat K$,
let $\tilde \alpha_{\hat k} \in \Cond(\T)$ denote the conditional torus element
\begin{equation}\label{tdiff}
 \tilde \alpha_{\hat k} \coloneqq \alpha_{\hat k} - w(\alpha_{\hat k}).
 \end{equation}
 Applying $w$ to both sides of \eqref{ckk} and subtracting, we conclude that
\begin{equation}\label{hka}
0 = \tilde \alpha_{\hat k_1 + \hat k_2} - \tilde \alpha_{\hat k_1} - \tilde \alpha_{\hat k_2}
\end{equation}
for all $\hat k_1, \hat k_2 \in \hat K$. By Theorem \ref{cond-pont}, we conclude that $(\tilde \alpha_{\hat k})_{\hat k \in \hat K}$ lies in $\Cond(\iota)(\Cond(K_\Baire))$, that is to say there exists $F \in \Cond(K_\Baire)$ such that
$$ \tilde \alpha_{\hat k} = \langle \hat k, F \rangle$$
for all $\hat k \in \hat K$.  On the other hand, from \eqref{hkr}, \eqref{tdiff} we have
$$  \langle \hat k, \rho_\gamma\rangle = \tilde \alpha_{\hat k} \circ T^\gamma - \tilde \alpha_{\hat k}
 $$
for all $\hat k \in K$ and $\gamma \in \Gamma$ and hence
\begin{equation}\label{hkb}
\langle \hat k, \rho_\gamma - (F \circ T^\gamma - F) \rangle = 0
\end{equation}
for all $\hat k \in \hat K$ and $\gamma \in \Gamma$.  Applying the injectivity claim of Theorem \ref{cond-pont}, we conclude that
$$ \rho_\gamma - (F \circ T^\gamma - F) = 0$$
for all $\gamma \in \Gamma$, and so $\rho$ is an abstract $K$-valued coboundary as required.

\section{Representing conditional elements of a space}\label{condrep-sec}

Throughout this section $X = (X, \Sigma_X, \mu)$ is assumed to be a measure space of positive measure.

If $Y = (Y,\Sigma_Y)$ is a concrete measurable space, and $f\colon X \to Y$ is a concrete measurable map, then the abstraction $[f] \in \Hom_{\AbsMes}(X_\mu; Y)=  \Cond(Y)$ defined in the introduction is a conditional element of $Y$, and can be defined explicitly as
$$ [f]^* E = [f^* E]$$
for $E \in \Sigma_Y$, where $[f^* E] \in X_\mu$ is the abstraction of $f^* E \in \Sigma_X$ in $X_\mu$.  Thus for instance $\Cond(c)$ is the abstraction of the constant function $x \mapsto c$ for all $c \in Y$.  It is clear that if $f,g\colon X \to Y$ are concrete measurable maps that agree $\mu$-almost everywhere, then $[f]=[g]$.  However, the converse is not true.  One trivial example occurs when $\Y$ fails to separate points:

\begin{example}[Non-uniqueness of realizations, I]  Let $Y = \{1,2\}$ with the trivial $\sigma$-algebra $\Sigma_Y = \{ \emptyset, Y \}$.  Then the constant concrete measurable maps $1$ and $2$  from $X$ to $Y$ are such that $[1]=[2]$, but $1$ is not equal to $2$ almost everywhere (if $X$ has positive measure).
\end{example}

However, there are also counterexamples when $\Sigma_Y$ does separate points, as the following example shows:

\begin{example}[Non-uniqueness of realizations, II]\label{bad-ex}  Let $X = [0,1]$ with Lebesgue measure $\mu$, and let $Y \coloneqq \{0,1\}^{[0,1]}$ with the product $\sigma$-algebra.  Let $f \colon X \to Y$ be the function defined by
$$ f(x) \coloneqq ( 1_{x=y} )_{y \in [0,1]}$$
for all $x \in [0,1]$, where the indicator $1_{x=y}$ equals $1$ when $x=y$ and zero otherwise, and let $g \colon X \to Y$ be the zero function $g(x) \coloneqq 0$.  Observe that $f(x) \neq g(x)$ for all $x \in [0,1]$, so $f$ and $g$ are certainly not equal almost everywhere.  However, the product $\sigma$-algebra in $Y = \{0,1\}^{[0,1]}$ is the union of the pullbacks of the $\sigma$-algebras on $\{0,1\}^I$ as $I$ ranges over at most countable subsets of $[0,1]$.  Thus if $E$ is measurable in $Y$, then $E = \pi_I^{-1}(E_I)$ for some measurable subset $E_I$ of $\{0,1\}^I$, where $\pi_I: \{0,1\}^{[0,1]} \to \{0,1\}^I$ is the projection map.  The function $\pi_I \circ f \colon X \to \{0,1\}^I$ is equal to $\pi_I \circ g = 0$ almost everywhere, thus $f^* E = (\pi_I \circ f)^*(E_I)$ is equal modulo null sets to $g^* E = (\pi_I \circ g)^* E_I$. We conclude that $[f]=[g]$, despite the fact that $f,g$ are not equal almost everywhere.
\end{example}

Note in the above example while $f$ and $g$ do not agree almost everywhere, each component of $f$ agrees with the corresponding component of $g$ almost everywhere, and it is the latter that allows us to conclude that $[f]=[g]$; this can also be derived from Proposition \ref{product}.  In particular, this example shows that the analogue of Proposition \ref{product} for the space $L^0(X; Y)$ of concrete measurable functions modulo almost everywhere equivalence fails.

For certain choices of $Y$, there exist conditional elements $y \in \Cond(Y)$ of $Y$ that are not represented by any concrete measurable map:

\begin{example}[Non-realizability]\label{nonrep}  Let $X = \mathrm{pt}$ be a point (with counting measure $\mu$), and let $Y \coloneqq \{0,1\}^{[0,1]} \backslash \{0\}^{[0,1]}$ be the product space $\{0,1\}^{[0,1]}$ with a point $\{0\}^{[0,1]}$ removed, endowed with the measurable structure induced from the product $\sigma$-algebra.  Observe that the point $\{0\}^{[0,1]} = \{0^{[0,1]}\}$ is not measurable in $\{0,1\}^{[0,1]}$ (all the measurable sets in this space are pullbacks of a measurable subset of $\{0,1\}^I$ for some countable $I \subset [0,1]$, and $\{0\}^{[0,1]}$ is not of this form).  Hence every measurable subset $E$ of $\{0,1\}^{[0,1]} \backslash \{0\}^{[0,1]}$ has a unique measurable extension $\tilde E$ to $\{0,1\}^{[0,1]}$.  Now let $y \in \Cond(Y)$ be the conditional element of $Y$ defined by
$$ y^* E = 1_{0^{[0,1]} \in \tilde E};$$
this is easily seen to be an element of $\Cond(Y)$.  However, it does not have any concrete realization $f \colon X \to Y$.  For if we had $y = [f]$, then we must have $1_{0^{[0,1]} \in \tilde E} = 1_{f(0) \in E}$ for every measurable subset $E$ of $\{0,1\}^{[0,1]}$. But $f(0) \in Y$ must have at least one coefficient equal to $1$, and is thus contained in a cylinder set $E$ whose extension $\tilde E$ does not contain $0^{[0,1]}$, a contradiction.
\end{example}

Nevertheless, we are able to locate some situations in which conditional elements of $Y$ are represented by concrete measurable maps.  From Proposition \ref{metr} we already can do this whenever $Y$ is a Polish space.  We can also recover a concrete realization of a conditional element of $K_\Baire$ in the case that $K$ is a compact Hausdorff abelian group.

\begin{proposition}[Conditional elements of compact abelian groups]\label{alt}  Let $K$ be a compact Hausdorff abelian group.  Then every conditional element $k \in \Cond(K_\Baire)$ has a realization by a concrete measurable map $f \colon X \to K_\Baire$.   
\end{proposition}

\begin{proof} Fix $K,k$.  Then $\langle \hat k, k \rangle \in \Cond(\T)$ for each $\hat k \in \hat K$ (where by abuse of notation we identify $\langle \hat k, \cdot \rangle$ with $\Cond(\langle \hat k, \cdot \rangle))$.  We will apply Zorn's lemma (in the spirit of the standard proof of the Hahn-Banach theorem) to the following setup.  Define a \emph{partial solution} to be a tuple $(G, (f_g)_{g \in G})$, where
\begin{itemize}
    \item $G$ is a subgroup of $\hat K$.
    \item For each $g \in G$, $f_g \colon G \to \T$ is a concrete measurable map with $[f_g] = \langle g, k \rangle$.
    \item For each $g_1,g_2 \in G$, one has $f_{g_1+g_2}(x) = f_{g_1}(x) + f_{g_2}(x)$ for \emph{every} $x \in X$ (not just $\mu$-almost every $x$).
\end{itemize}
We place a partial order on partial solutions by setting $(G, (f_g)_{g \in G}) \leq (G', (f'_{g'})_{g' \in G'})$ if $G \leq G'$ and $f_g = f'_g$ for all $g \in G$.  Since $(\{0\}, (0))$ is a partial solution, and every chain of partial solutions has an upper bound, we see from Zorn's lemma that there exists a maximal partial solution $(G, (f_g)_{g \in G})$.  We claim that $G$ is all of $\hat K$.  Suppose this is not the case, then we can find an element $\hat k$ of $\hat K$ that lies outside of $G$.  There are two cases, depending on whether $n \hat k \in G$ for some natural number $n$.

First suppose that $n \hat k \not \in G$ for all $n \in \N$.  By Proposition \ref{metr}, we can find a concrete measurable map $f_{\hat k} \colon X \to \T$ such that $[f_{\hat k}] = \langle \hat k, k \rangle$.  We then define $f_{n \hat k + g} \colon X \to \T$ for all $n \in \Z \backslash \{0\}$ and $g \in G$ by the formula
\begin{equation}\label{nkform}
 f_{n \hat k + g}(x) \coloneqq n f_{\hat k}(x) + f_g(x).
 \end{equation}
 If we set 
 \begin{equation}\label{gp-def}
 G' = \{ n \hat k + g: n \in \Z, g \in G \}
 \end{equation}
 to be the group generated by $\hat k$ and $G$, we can easily check that $(G', (f_{g'})_{g' \in G})$ is a partial solution that is strictly larger than $(G, (f_g)_{g \in G})$, contradicting maximality.

Now suppose that there is a least natural number $n_0$ such that $n_0 \hat k \in G$.  We can find a concrete measurable map $\tilde f_{\hat k} \colon X \to \T$ such that $[\tilde f_{\hat k}] = \langle \hat k, k \rangle$.  This map cannot immediately be used as our candidate for $f_{\hat k}$ because it does not necessarily obey the consistency condition $n_0 \tilde f_{\hat k}(x) = f_{n_0 \hat k}(x)$ for all $x \in X$.  However, this identity is obeyed for \emph{almost all} $x \in X$.  Let $N$ be the null set on which the identity fails.  We then set $f_{\hat k}(x)$ to equal $\tilde f_{\hat k}(x)$ when $x \not \in N$ and equal to $g_{n_0}( f_{n_0 \hat k}(x) )$ when $x \in N$, where (as in the previous section) $g_{n_0}: \T \to \T$ is a measurable map for which $n_0 g_{n_0}(\theta) = \theta$ for all $\theta \in \T$.  Then $[f_{\hat k}] = [\tilde f_{\hat k}] = \langle \hat k, k \rangle$.  If one then defines $f_{n \hat k + g}$ for all $ n \in \Z$ and $g \in G$ by the same formula as before, we see that this is a well defined formula for $f_{g'}$ for all $g'$ in the group \eqref{gp-def}, and that $(G', (f_{g'})_{g' \in G})$ is a partial solution that is strictly larger than $(G, (f_g)_{g \in G})$, again contradicting maximality.  This completes the proof that $G = \hat K$.

By Pontryagin duality \eqref{iotak}, for each $x \in X$ there is a unique element $f(x) \in K$ such that $f_{\hat k}(x) = \langle \hat k, f(x) \rangle$ for all $\hat k \in \hat K$.  This gives a map $f \colon X \to K_\Baire$; as all the maps $\langle k, f \rangle = f_{\hat k}$ are measurable, we see that $f$ is also measurable as the $\sigma$-algebra of $K_\Baire$ is generated by the characters $\hat k$.  From Theorem \ref{cond-pont} we see that $[f]=k$, and the claim follows.
\end{proof}

One can ask if the proposition holds for all compact Hausdorff spaces, not just the compact Hausdorff abelian groups.  We were unable\footnote{We thank the referee for pointing out a serious error in the results claimed in this direction in a previous version of this manuscript.} to make significant headway on this question, but can at least treat the simple case when the base space $X$ is atomic:

\begin{lemma}[The case of an atomic space]\label{simp}  Let $K$ be a compact Hausdorff space and suppose that $X$ is a $\sigma$-finite atomic measure space.  Then every element of $\Cond(K_\Baire)$ is represented by a concrete measurable map from $X$ to $K_\Baire$, unique up to almost everywhere equivalence.
\end{lemma}

Note that Example \ref{nonrep} shows that the requirement that $K$ be compact cannot be completely omitted in this lemma.

\begin{proof}  By contracting all atoms in $X$ down to points and removing all null sets, we may assume without loss of generality that $X$ is countable and discrete, with all points having positive measure. (In particular $X$ has no non-trivial null sets, and all functions on $X$ are measurable.)

From Theorem \ref{weil-thm} we see that any two distinct functions $F, F' \colon X \to K$ are separated at at least one point $x \in X$ by preimages of disjoint balls with respect to a continuous map $\pi\colon K \to S$ into a metric space, and hence are also distinct as elements of $\Cond(K_\Baire)$ as such preimages are measurable and every point in $X$ has positive measure. This gives uniqueness.  It remains to show that every conditional element $k \in \Cond(K_\Baire)$ of $K_\Baire$  arises from a function from $X$ to $K$.  By Theorem \ref{weil-thm}, we may assume that $K_\Baire$ is a closed subset of $S_A = \prod_{\alpha \in A} S_\alpha$ for some metric spaces $S_\alpha$, with the product $\sigma$-algebra.  For each $\alpha \in A$, let $\pi_\alpha \colon K_\Baire \to S_\alpha$ be the coordinate map, then $\pi_\alpha(k) \in \Cond(S_\alpha)$. By Proposition \ref{metr} there is a unique function $s_\alpha\colon X \to S_\alpha$ such that $\pi_\alpha(k) = [s_\alpha]$.  If we set $s \colon X \to S_A$ to be the tuple $s \coloneqq (s_\alpha)_{\alpha \in A}$, then by Proposition \ref{product} we have $k = [s]$.  By Proposition \ref{condk-desc}, this implies that $\pi_I(s)$ takes values everywhere in $\pi_I( K )$ for all countable $I \subset A$, and hence by the closed nature of $K$ we see that $s$ takes values in $K$ everywhere.  Thus $k$ has a representation as a measurable map from $X$ to $K_\Baire$ as required.
\end{proof}

\section{Towards a concrete version of the uncountable Moore--Schmidt theorem}\label{concrete-sec}

One can raise the conjecture of whether Theorem \ref{mstu} continues to hold if we use concrete actions, coboundaries, and cocycles:

\begin{conjecture}[Concrete uncountable Moore--Schmidt conjecture]\label{mstu-concrete}
Let $\Gamma$ be a discrete group acting concretely on a measure space $X = (X,\Sigma_X,\mu)$, and let $K$ be a compact Hausdorff abelian group.  Then a concrete $K_\Baire$-valued cocycle $\rho = (\rho_\gamma)_{\gamma \in \Gamma}$ on $X$ is an concrete coboundary if and only if the $\T$-valued concrete cocycles $\hat k \circ \rho \coloneqq (\hat{k}\circ \rho_\gamma)_{\gamma \in \Gamma}$ are concrete coboundaries for all $\hat{k} \in \hat K$.
\end{conjecture}

The ``only if'' part of the conjecture is easy; the difficulty is the ``if'' direction.  If $\rho = (\rho_\gamma)_{\gamma \in \Gamma}$ is a concrete coboundary with the property that $\hat k \circ \rho$ is a concrete coboundary for all $\hat k \in \hat K$, then the abstraction $[\rho] \coloneqq ([\rho_\gamma])_{\gamma \in \Gamma}$ is clearly an abstract coboundary with $\hat k \circ [\rho] = [\hat k \circ \rho]$ an abstract coboundary for all $\hat k \in \hat K$.  Applying Theorem \ref{mstu}, we conclude that $[\rho]$ is an abstract coboundary, thus there exists an abstract measurable map $F \in \Hom_{\AbsMes}(X_\mu; K_\Baire)$ such that
$$ [\rho_\gamma] = F \circ T^\gamma - F$$
for all $\gamma \in \Gamma$.  By Proposition \ref{alt}, we may then find a concrete measurable map $\tilde F \colon X \to K_\Baire$ such that $[\tilde F] = F$.  If we then introduce the concrete coboundary 
$$ \tilde \rho \coloneqq ( \tilde F \circ T^\gamma - \tilde F )_{\gamma \in \Gamma}$$
then we see that $[\rho] = [\tilde \rho]$.  If we could conclude that $\rho = \tilde \rho$, we could establish Conjecture \ref{mstu-concrete}. We are unable to do this, but by subtracting $\tilde \rho$ from $\rho$ we see that to prove the above conjecture it suffices to do so in the case $\tilde \rho=0$, which implies that $[\langle \hat k, \rho_\gamma \rangle] = 0$, or equivalently (by Proposition \ref{metr}) that $\langle \hat k, \rho_\gamma \rangle$ vanishes almost everywhere for each $\hat k, \gamma$.  Thus Conjecure \ref{mstu-concrete} can be equivalently formulated as

\begin{conjecture}[Concrete uncountable Moore--Schmidt conjecture, reduced version]\label{mstu-concrete-red}
Let $\Gamma$ be a discrete group acting concretely on a measure space $X = (X,\Sigma_X,\mu)$, and let $K$ be a compact Hausdorff abelian group.  Let $\rho = (\rho_\gamma)_{\gamma \in \Gamma}$ be a concrete $K_\Baire$-valued cocycle on $X$ with the property that $\langle \hat k, \rho_\gamma \rangle$ vanishes $\mu$-almost everywhere for each $\hat k \in \hat K$ and $\gamma \in \Gamma$.  Then $\rho$ is a concrete coboundary.
\end{conjecture}

One easily verified case of this conjecture is when $K$ is metrizable.  Then $\hat K$ is countable, so for each $\gamma \in \Gamma$ we see that for almost every $x \in X$, $\langle \hat k, \rho_\gamma(x) \rangle=0$ for all $\hat k \in \hat K$ simultaneously, and so $\rho_\gamma(x)=0$ for almost every $x$, which of course implies that $\rho$ is a coboundary.  Note that this allows us to recover Theorem \ref{mst} from Theorem \ref{mstu}.

Another easy case is when $\Gamma$ is countable, $(X,\Sigma_X,\mu)$ is complete, and $K$ is a torus $K = \T^A$ for some (possibly uncountable) $A$.  By hypothesis, the cocycle equation
\begin{equation}\label{cocycle} \rho_{\gamma_1\gamma_2}(x) = \rho_{\gamma_1} \circ T^{\gamma_2}(x) + \rho_{\gamma_2}(x)
\end{equation}
holds for each $\gamma_1,\gamma_2 \in \Gamma$ for $x$ outside of a null set.  Since $\Gamma$ is countable, we may make this null set independent of $\gamma_1,\gamma_2$, and can also make it $\Gamma$-invariant.  We may then delete this set from $X$ and assume without loss of generality that \eqref{cocycle} holds for \emph{all} $x \in X$.  Now we write $\rho$ in coordinates as $\rho_\gamma(x) = (\rho_{\gamma,\alpha}(x))_{\alpha \in A}$.  Then for each $\alpha \in A$, $\rho_{\gamma,\alpha}(x)$ vanishes for $x$ outside of a null set $N_\alpha$, which as before we can assume to be independent of $\gamma$ and $\Gamma$-invariant.  By the axiom of choice, we may partition $N_\alpha$ into disjoint orbits of $\Gamma$:
$$ N_\alpha = \bigcup_{x \in M_\alpha} \{ T^\gamma x: \gamma \in \Gamma \}$$
where $M_\alpha$ is a subset of $N_\alpha$.  If we then define the map $F_\alpha\colon X \to \T$ by setting
$$ F_\alpha( T^\gamma x ) \coloneqq \rho_{\gamma,\alpha}(x)$$
for $x \in M_\alpha$ and $\gamma \in \Gamma$, and $F_\alpha(x) = 0$ for $x \not \in N_\alpha$, then by the completeness of $(X,\X,\mu)$ we see that $F_\alpha$ is measurable (being zero almost everywhere) and from the cocycle equation we see that
$$ \rho_{\gamma,\alpha}(x) = F_\alpha(T^\gamma x) - F_\alpha(x)$$
for all $x \in X$, $\gamma \in \Gamma$, $\alpha \in A$.  Setting $F: X \to K_\Baire$ to be the map $F(x) \coloneqq (F_\alpha(x))_{\alpha \in A}$, we conclude that $\rho_\gamma(x) = F(T^\gamma(x)) - F(x)$ for all $\gamma \in \Gamma$ and $x \in X$, so that $\rho$ is a concrete coboundary as claimed in this case. 

It is conceivable that the truth of this conjecture is sensitive to undecidable axioms in set theory.

\appendix

\section{A counterexample to a general product theorem for conditional elements}\label{counter}

In this appendix we establish

\begin{proposition}[Counterexample to general product theorem]\label{counter-prop} Let $(X,\Sigma_X,\mu)$ be the unit interval $[0,1)$ with the Borel $\sigma$-algebra $\Sigma_X$ and Lebesgue measure $\mu$.  Then there exist concrete measurable spaces $(Y_1,\Sigma_{Y_1}),(Y_2,\Sigma_{Y_2})$ and conditional elements $y_1 \in \Cond(Y_1), y_2 \in \Cond(Y_2)$ such that there does not exist any conditional element $y \in \Cond(Y_1 \times Y_2)$  with $\pi_1(y)=y_1$ and $\pi_2(y)=y_2$, where $\pi_i\colon Y_1 \times Y_2 \to Y_i$ are the coordinate projections for $i=1,2$.
\end{proposition}

In particular, this proposition demonstrates that the equality
$$ \Cond\left(Y_1\times Y_2\right) = \Cond(Y_1)\times \Cond(Y_2)$$ 
can fail without further hypotheses on $Y_1,Y_2$, such as being a Polish space (as in Proposition \ref{product}) or compact Hausdorff with the Baire $\sigma$-algebra (as in Corollary \ref{square}).  This proposition is not required to prove any of the other results in this paper.

To construct $Y_1,Y_2$ we use

\begin{lemma}[Disjoint sets of full outer measure]\label{outer} There exist disjoint subsets $Y_1,Y_2 \subset X$ such that $Y_1,Y_2$ both have outer measure $1$. (In particular, every subset of $X$ of positive measure has a non-empty intersection with both $Y_1$ and $Y_2$.)
\end{lemma}

Of course, any sets $Y_1,Y_2$ obeying the conclusions of this lemma are necessarily nonmeasurable.
\begin{proof} We partition $X$ into Vitali equivalence classes $X \cap (x+\Q)$ for $x \in \R$.  As Borel sets of $X$ have the cadinality $2^{\aleph_0}$ of the continuum, we may well-order them as $(A_\beta)_{\beta<2^{\aleph_0}}$, where $\beta$ ranges over all ordinals of cardinality less than that of the continuum. By an alternating transfinite recursion\footnote{We learned of this construction from {\tt math.stackexchange.com/questions/157532}.}, construct two disjoint sets $Y_1=\{x_\beta\colon \beta<2^{\aleph_0}\}$ and $Y_2=\{y_\beta\colon \beta<2^{\aleph_0}\}$ such that 
\begin{itemize}
    \item[(i)]  $x_\beta\neq y_\beta$ and $x_\beta$ is not in the same Vitali equivalence class of $x_\gamma$ for $\gamma<\beta$ and similarly $y_\beta$ is not in the same Vitali equivalence class of $y_\gamma$ for $\gamma<\beta$. 
    \item[(ii)]  $x_\beta,y_\beta\in A^c_\beta$ whenever $A^c_\beta$ is uncountable. 
\end{itemize}
One can always select $x_\beta,y_\beta$ at each stage of the recursion because uncountable Borel (or analytic) sets contain perfect sets and hence have cardinality $2^{\aleph_0}$, see e.g., \cite[Theorem 29.1]{kechris1995classical}.  By construction, for any Borel set $A$ such that $Y_1\subset A$ or $Y_2 \subset A$ it follows that $A^c$ is countable, implying that $Y_1,Y_2$ have outer measure $1$.  
\end{proof}

Let $Y_1,Y_2$ be as in the above lemma.  Let ${\mathcal A}$ be the Boolean algebra of $X$ generated by the half-open dyadic intervals $[\frac{j}{2^n}, \frac{j+1}{2^n})$ in $X$, and for $i=1,2$, let $\Sigma_{Y_i}$, ${\mathcal A}_i$ be the restrictions of $\Sigma_X$, ${\mathcal A}$ respectively to $Y_i$.  Clearly each $(Y_i,\Sigma_{Y_i})$ is a concrete measurable space.  Since ${\mathcal A}$ generates $\Sigma_X$ as a $\sigma$-algebra, we see that ${\mathcal A}_i$ generates $\Sigma_{Y_i}$ as a $\sigma$-algebra; also, as $Y_i$ has full outer measure and is therefore dense in $X$, we see that each $A \in {\mathcal A}_i$ arises as $\phi_i(A) \cap Y_i$ for a unique $\phi_i(A) \in {\mathcal A}$. One then easily verifies that $\phi_i \colon {\mathcal A}_i \to \mathcal{A}$ is a Boolean algebra isomorphism.  We have the following key property:

\begin{lemma}[Weak $\sigma$-homomorphism]\label{wca}  Let $i=1,2$.  If $(A_n)_{n \in \N}$ are a family of pairwise disjoint sets in ${\mathcal A}_i$ with $\bigcup_{n=1}^\infty A_n \in {\mathcal A}_i$, then the sets $\bigcup_{n=1}^\infty \phi_i(A_n)$ and $\phi_i( \bigcup_{n=1}^\infty A_n )$ differ by a set of measure zero.  
\end{lemma}

\begin{proof}  For each $n$, let $B_n \coloneqq \bigcup_{m=1}^\infty A_m \backslash \bigcup_{m=1}^{n-1} A_m \in {\mathcal A}_i$.  The set $\bigcap_{n=1}^\infty \phi_i(B_n)$ is a Borel measurable subset of $X$.  If it has positive measure, then by Lemma \ref{outer}, it intersects $Y_i$ in at least one point $y$; as $B_n = \phi_i(B_n) \cap Y_i$, we conclude that $y$ lies in each of the $B_n$, which contradicts the fact that $\bigcap_{n=1}^\infty B_n = \emptyset$.  Thus $\bigcap_{n=1}^\infty \phi_i(B_n)$ has measure zero; since $\phi_i( \bigcup_{n=1}^\infty A_n )$ is the disjoint union of $\bigcup_{n=1}^\infty \phi_i( A_n )$ and $\bigcap_{n=1}^\infty \phi_i(B_n)$, we obtain the claim.
\end{proof}

We combine this lemma with the following general extension theorem, which may be of independent interest:

\begin{proposition}[Extension theorem]\label{extension}  Let $(Y,\Sigma_Y)$ be a concrete measurable space, with $\Sigma_Y$ generated by a Boolean algebra $\mathcal{A}$.  Let $(X,\Sigma_X,\mu)$ be a finite measure space, and let $\alpha \colon \mathcal{A} \to X_\mu$ be a Boolean algebra homomorphism.  Then the following are equivalent:
\begin{itemize}
    \item[(i)] (extension to $\sigma$-algebra homomorphism) There exists a unique extension of $\alpha$ to a $\sigma$-complete Boolean algebra homomorphism $\tilde \alpha \colon \Sigma_Y \to X_\mu$.
    \item[(ii)] (weak $\sigma$-homomorphism property) If $(A_n)_{n \in \N}$ are a family of pairwise disjoint sets in ${\mathcal A}$ with $\bigcup_{n=1}^\infty A_n \in {\mathcal A}$, then one has 
\begin{equation}\label{premeasure}
\bigvee_{n=1}^\infty \alpha(A_n) = \alpha\left( \bigcup_{n=1}^\infty A_n \right).
\end{equation}
\end{itemize}  
\end{proposition}

\begin{proof}  Clearly (i) implies (ii).  Now assume (ii).  The uniqueness of a $\sigma$-complete Boolean algebra homomorphism is clear since $\mathcal{A}$ generates $\Sigma_Y$, so we focus on existence.  By Example \ref{pointless}, $X_\mu$ (viewed as a measure algebra) is not necessarily representable as a $\sigma$-algebra of sets. So we cannot apply the $\sigma$-complete version of the Sikorski extension theorem, see \cite[\S~34]{sikorski2013boolean}.  Instead, we appeal to an extension theorem for vector-valued measures\footnote{See \cite{joseph1977vector} for any unexplained definition or result in the theory of vector measures.}, viewing a $\sigma$-complete Boolean algebra (resp. Boolean algebra) homomorphism as a special type of vector-valued countably additive (resp. finitely additive) measure.  Indeed, observe that $X_\mu$ (viewed as a measure algebra) comes with a natural complete metric $d(a,b) \coloneqq \mu(a\Delta b)$, and therefore can be embedded as a metric space into $L^1(X_\mu)$ by identifying each abstractly measurable subset $a$ of $X_\mu$ with its indicator function $1_a \in L^1(X_\mu)$.  Here $L^1(X_\mu)$ denotes the Banach space of absolutely integrable (abstractly) measurable functions from $X_\mu$ to $\R$ (which can also be identified with the absolutely integrable concretely measurable functions from $(X,\Sigma_X,\mu)$ to $\R$ modulo almost everywhere equivalence, see \cite{fremlin1989measure}).

The map $F \colon \mathcal{A} \to L^1(X_\mu)$ defined by $F(A) \coloneqq 1_{\alpha(A)}$ is a finitely additive  vector measure which is strongly continuous\footnote{That is, $\sum_{n=1}^\infty F(A_n)$ converges in norm whenever $(A_n)$ are pairwise disjoint sets in $\mathcal{A}$.}. By the Carath{\'e}odory-Hahn-Kluvanek extension theorem for vector measures \cite[\S~I.5]{joseph1977vector}, $F$ will have an extension to a countably additive vector measure on $(Y,\Sigma_Y)$ if $F$ is weakly countably additive, that is it obeys the premeasure property $\langle F(\bigcup_{n=1}^\infty A_n),f\rangle =\sum_{n=1}^\infty \langle F(A_n),f\rangle$ (where $\langle \cdot,\cdot \rangle$ denotes the duality pairing between $L^1(X_\mu)$ and $L^\infty(X_\mu)$) for every $f\in L^\infty(X_\mu)$ and every countable family $(A_n)$ of pairwise disjoint sets in $\mathcal{A}$ such that $\bigcup_{n=1}^\infty A_n\in \mathcal{A}_i$.  But this property follows from \eqref{premeasure}, which implies in particular that $\sum_{n=1}^\infty F(A_n)$ converges strongly in $L^1(X_\mu)$ to $F(\bigcup_{n=1}^\infty A_n)$.  Thus we have a countably additive extension $\tilde F \colon \Sigma_Y \to L^1(X_\mu)$.  If $A \in \Sigma_Y$, then $\tilde F(A)$ is necessarily an indicator function $1_{\tilde \alpha(A)}$ in $L^1(X_\mu)$ for some abstractly measurable subset $\tilde \alpha(A) \in X_\mu$ of $X_\mu$, because $\tilde F$ is constructed as a metric extension of a uniformly  continuous function on the dense set $(\mathcal{A},d_\nu)$ where $d_\nu$ is a metric associated to a countably additive finite measure $\nu$ on $\Sigma_Y$ (see the proof of \cite[\S~I.5, Theorem 2]{joseph1977vector} for details).  The map $\tilde \alpha \colon \Sigma_Y \to X_\mu$ then gives the required extension.  
\end{proof}

For $i=1,2$, we apply Proposition \ref{extension} to the Boolean algebra homomorphism $\alpha_i \colon \mathcal{A}_i\to \X_\mu$ defined by $\alpha_i (A) \coloneqq [\phi_i(A)]$ for any $A \in {\mathcal A}_i$. By Lemma \ref{wca}, the property in Proposition \ref{extension}(ii) holds, and thus we can extend $\alpha_i$ to a $\sigma$-complete Boolean algebra homomorphism $\tilde \alpha_i\colon \Y_i\to \X_\mu$, thus $y_i \coloneqq \tilde \alpha_i^\op$ is a conditional element of $Y_i$ for $i=1,2$.  
Now suppose for sake of contradiction that there was a conditional element $y \in \Cond(Y_1 \times Y_2)$ with $\pi_1(y)=y_1$ and $\pi_2(y)=y_2$.  Then for every dyadic interval $I$, we have
$$ y^*( (Y_1 \cap I) \times Y_2 ) = y_1^*(Y_1 \cap I) = \tilde \alpha_1(Y_1 \cap I) = \alpha_1(Y_1 \cap I) = [I]$$
and similarly
$$ y^*( Y_1 \times (Y_2 \cap I) ) = [I]$$
and hence
$$ y^*( (Y_1 \times Y_2) \cap (I \times I) ) = [I].$$
Letting $I$ range over the dyadic intervals of length $\mu(I) = 2^{-n}$ for a given natural number $n$, we conclude that
$$ y^*\left( (Y_1 \times Y_2) \cap \bigcup_{I: \mu(I)=2^{-n}} (I \times I) \right) = 1.$$
Taking intersections in $n$, we conclude that
$$ y^*( (Y_1 \times Y_2) \cap \{ (x,x): x \in X \} ) = 1.$$
But as $Y_1, Y_2$ are disjoint, the intersection $(Y_1 \times Y_2) \cap \{ (x,x): x \in X \}$ is empty.  This contradiction establishes Proposition \ref{counter-prop}.

We close this appendix with a further application of Proposition \ref{extension}, in the spirit of Corollary \ref{square}.

\begin{proposition}[Conditional elements of product spaces, III]\label{half-square}  Let $X = (X,\Sigma_X,\mu)$ be a probability space, let $Y = (Y, \Sigma_Y)$ be a concrete measurable space, and let $K$ be a compact Hausdorff space.  Then $\Cond(Y \times K_\Baire) = \Cond(Y) \times \Cond(K_\Baire)$.
\end{proposition}

\begin{proof} We need to show that for any $y \in \Cond(Y)$ and $k \in \Cond(K_\Baire)$ there exists a unique $\sigma$-complete Boolean homomorphism $\alpha \colon \Sigma_Y \otimes \Baire(K) \to X_\mu$ such that $\alpha(E) = y^*(E)$ for all $E \in \Sigma_Y$ and $\alpha(F) = k^*(F)$ for all $F \in \Baire(K)$, where we view $\Sigma_Y$ and $\Baire(K)$ as subalgebras of the $\sigma$-algebra $\Sigma_Y \otimes \Baire(K)$.

Let ${\mathcal A}$ be the Boolean subalgebra of $\Sigma_Y \otimes \Baire(K)$ whose elements consist of finite disjoint unions of ``rectangles'' $E \times F$ where $E \in \Sigma_Y$, $F \in \Baire(K)$.  Clearly there is a unique Boolean algebra homomorphism $\alpha \colon {\mathcal A} \to X_\mu$ such that $\alpha(E \times F) = y^*(E) \wedge k^*(F)$ for any $E \in \Sigma_Y$, $F \in \Baire(K)$.  Since ${\mathcal A}$ generates $\Sigma_Y \otimes \Baire(K)$ as a $\sigma$-algebra, it suffices by Proposition \ref{extension} to show that whenever $(A_n)_{n \in \N}$ are a family of disjoint subsets of ${\mathcal A}$ such that $\bigcup_{n=1}^\infty A_n \in {\mathcal A}$, that
$$ \alpha\left( \bigcup_{n=1}^\infty A_n \right) = \bigvee_{n=1}^\infty \alpha(A_n).$$
By adding the complement of $\bigcup_{n=1}^\infty A_n$ to the $A_n$, we may assume that $\bigcup_{n=1}^\infty A_n = Y \times K$.  By breaking up each $A_n$ into rectangles we may assume that $A_n = E_n \times F_n$ with $E_n \in \Sigma_Y$ and $F_n \in \Baire(K)$.  Thus the $E_n \times F_n$ form a partition of $Y \times K$, and it suffices to show that
$$ \bigvee_{n=1}^\infty y^*(E_n) \wedge k^*(F_n) = 1.$$
By definition of $X_\mu$, it suffices to show that
$$ \mu\left( \bigvee_{n=1}^\infty y^*(E_n) \wedge k^*(F_n) \right) \geq 1 - \eps$$
for any $\eps>0$.

Fix $\eps$. By definition of the Baire $\sigma$-algebra, each $F_n$ lies in the $\sigma$-algebra generated by a continuous map to a compact metric space; since the product of countably many compact metric spaces is metrizable, we can place all the $F_n$ in a $\sigma$-algebra generated by a continuous map to a single compact metric space $S$.  We can then push forward $K$ to $S$, thus we may assume without loss of generality that $K$ is a compact metric space, so $\Baire(K)$ is now the Borel $\sigma$-algebra.  The pushforward measure $k_* \mu$ is then a Borel probability measure on the compact metric space $K$, and hence regular (see e.g.~\cite[Theorem 1.1]{billingsley2013convergence}). In particular, we can find an open neighborhood $U_n$ of $F_n$ in $K$ for each $n$ such that
$$y^*(U_n \backslash F_n) \leq \frac{\eps}{2^n}$$
and so it will suffice to show that
$$ \mu\left( \bigvee_{n=1}^\infty y^*(E_n) \wedge k^*(U_n) \right) \geq 1.$$
By construction, we have
$$ \bigcup_{n=1}^\infty E_n \times U_n = Y \times K.$$
Equivalently, for each $y \in Y$, the sets $\{ U_n: y \in E_n \}$ form an open cover of $K$.  As $K$ is compact, we thus see that for each $y \in Y$ there exists a finite subset $I \subset \{ n \in \N: y\in E_n \}$ such that $\bigcup_{n \in I} U_n = K$.  To put this another way, if we let ${\mathcal F}$ denote the collection of all finite subsets $I \subset \N$ with $\bigcup_{n \in I} U_n = K$, then we have
$$ \bigcup_{I \in {\mathcal F}} \bigcap_{n \in I} E_n = Y.$$
As ${\mathcal F}$ is at most countable, we can totally order it so that every element has finitely many predecessors.  If for each $I \in {\mathcal F}$ we set
$$ E'_I \coloneqq \bigcap_{n \in I} E_n \backslash \bigcup_{J<I} \bigcap_{n \in J} E_J$$
then the $E'_I$ form an at most countable partition of $Y$ into measurable sets, hence the $y^*(E'_I)$ are an at most countable partition of $1$ in $X_\mu$.  It thus suffices to show that
$$ \mu\left( \bigvee_{n=1}^\infty y^*(E_n) \wedge k^*(U_n) \wedge y^*(E'_I) \right) \geq \mu(y^*(E'_I))$$
for every $I$.  But we have
$$ \bigvee_{n=1}^\infty y^*(E_n) \wedge k^*(U_n) \wedge y^*(E'_I) \geq  \bigvee_{n \in I} k^*(U_n) \wedge y^*(E'_I) \geq y^*(E'_I)$$
since the $U_n, n \in I$ are a finite cover of $K$, and the claim follows.
\end{proof}

\bibliographystyle{abbrv}
\bibliography{bibliography}

\end{document}